\documentclass[12pt]{amsart}
\usepackage{mathptmx}
\usepackage{helvet}
\usepackage{courier}
\usepackage[T1]{fontenc}
\usepackage[latin9]{inputenc}
\usepackage{geometry}
\usepackage{graphicx}
\usepackage{amssymb}

\makeatletter

 \theoremstyle{definition}
 \newtheorem{defn}{Definition}
  \theoremstyle{remark}
  \newtheorem{rem}{Remark}
\newenvironment{lyxlist}[1]
{\begin{list}{}
{\settowidth{\labelwidth}{#1}
 \setlength{\leftmargin}{\labelwidth}
 \addtolength{\leftmargin}{\labelsep}
 }}
{\end{list}}
  \theoremstyle{plain}
  \newtheorem{lem}{Lemma}
  \theoremstyle{plain}
  \newtheorem{cor}{Corollary}
  \theoremstyle{plain}
  \newtheorem{prop}{Proposition}

\usepackage{hyperref}

\usepackage{stmaryrd}
\renewcommand{\subsection}{\@startsection%
 {subsection}%
 {2}%
 {0mm}%
 {0.5\baselineskip}%
 {\fontdimen4\font}%
 {\normalfont\bfseries}}
\newcommand{\Arg}{\textrm{Arg}}
\newcommand{\Tr}{\textrm{Tr}}

\newcommand{\PSLR}{\text{PSL}_{2}(\ensuremath{\mathbb{R}})}
\newcommand{\PGLR}{\text{PGL}_{2}(\ensuremath{\mathbb{R}})}
\newcommand{\PSLZ}{\text{PSL}_{2}(\ensuremath{\mathbb{Z}})}
\newcommand{\SLR}{\text{SL}_{2}(\ensuremath{\mathbb{R}})}
\newcommand{\SM}{\ensuremath{\mathcal{T}^{1}\mathcal{M}}}
\newcommand{\SH}{\ensuremath{\mathcal{T}^{1}\H}}

\newcommand{\sign}{\textrm{sign}}
\newcommand{\sgn}{\textrm{sign}}
\newcommand{\Fq}{F_{q}}
\newcommand{\Fqd}{F_{q}^{*}}
\newcommand{\Fqx}{\tilde{F}_{q}}
\newcommand{\Sx}{\tilde{S}}
\newcommand{\np}{\kappa}
\newcommand{\bi}{\ensuremath{\mathfrak{C}}}
\newcommand{\Bi}{\ensuremath{\tilde{\mathfrak{C}}}}
\newcommand{\Cq}{\ensuremath{\mathfrak{c}_{q}}}
\newcommand{\Cd}{\ensuremath{\mathfrak{c}_{q}^{*}}}

\renewcommand{\H}{\ensuremath{\mathcal{H}}}
\usepackage[all]{xy}
\newcommand{\FS}[1]{\ensuremath{\llbracket #1 \rrbracket}}

\newcommand{\RS}[1]{\ensuremath{\left[ \smash{#1}\vphantom{1^{j}} \right] }}

\newcommand{\DS}[1]{\ensuremath{\left[ \smash{#1}\vphantom{1^{j}} \right]^{*} }}
\newcommand{\RNCF}{\ensuremath{{\mathcal{A}}_{q} }}
\newcommand{\DRNCF}{\ensuremath{{\mathcal{A}}_{q}^{*} }}
\newcommand{\ZRNCF}{\ensuremath{{\mathcal{A}}_{0,q} }}
\newcommand{\ZDRNCF}{\ensuremath{\smash{{\mathcal{A}}_{0,q}^{*}} }}

\newcommand{\BIRNCF}{\ensuremath{{\mathcal{B}}_{q} }}
\newcommand{\BIRED}{\ensuremath{{\mathcal{B}}_{q}^{\circ} }}
\newcommand{\SBIRED}{\ensuremath{ \mathcal{B}_{q,s}^{\circ} } } 

\newcommand{\PT}{\ensuremath{\tilde{\mathcal{P}}}}
\newcommand{\PP}{\ensuremath{\mathcal{P}}}
\newcommand{\RET}{\ensuremath{\mathcal{T}}}
\newcommand{\RETx}{\ensuremath{\tilde{\mathcal{T}}}}
\newcommand{\Sred}{\ensuremath{\Omega^{*}_{s}}}
\newcommand{\Sredx}{\ensuremath{\Omega^{\infty}_{s}}}
\newcommand{\Sredo}{\ensuremath{\Omega_{s}}}
\newcommand{\Sgeo}{\ensuremath{\Upsilon_{s}}}
\newcommand{\Rgeo}{\ensuremath{{\Upsilon}}}
\newcommand{\nlm}[1]{\ensuremath{ \left\{ #1 \right\} _{\lambda}  } }
\newcommand{\nlms}[1]{\ensuremath{ \left\{ #1 \right\} _{\lambda} ^{*} } }
\newcommand{\K}{\ensuremath{\mathcal{K}}}
\newcommand{\n}{\ensuremath{\mathfrak{n}}}
\author[D. Mayer]{Dieter Mayer}
\address{Institut f\"ur Theor. Physik, TU Clausthal,
Abt. Statistische Physik und Nichtlineare Dynamik,
Arnold Sommerfeld Stra{\ss}e 6,
38678 Clausthal-Zellerfeld, Germany}
\email{dieter.mayer@tu-clausthal.de}

\author[F. Str{\"o}mberg]{Fredrik Str{\"o}mberg}
\address{Institut f\"ur Theor. Physik, TU Clausthal,
Abt. Statistische Physik und Nichtlineare Dynamik,
Arnold Sommerfeld Stra{\ss}e 6,
38678 Clausthal-Zellerfeld, Germany}
\email{fredrik.stroemberg@tu-clausthal.de}

\subjclass{Primary: 37D40; Secondary: 37E05,11A55,11K50}
\keywords{Symbolic dynamics, Geodesic flow, Hecke triangle groups, $\lambda$-fractions}

\makeatother

\begin{document}

\title{Symbolic dynamics for the geodesic flow on Hecke surfaces}

\begin{abstract}
In this paper we discuss a coding and the associated symbolic dynamics
for the geodesic flow on Hecke triangle surfaces. We construct an
explicit cross section for which the first return map factors through
a simple (explicit) map given in terms of the generating map of a
particular continued fraction expansion closely related to the Hecke
triangle groups. We also obtain explicit expressions for the associated
first return times. 

\tableofcontents{}
\end{abstract}
\maketitle

\section{Introduction}

Surfaces of negative curvature and their geodesics have been studied
since the 1898 work of Hadamard \cite{29.0522.01} (see in particular
the remark at the end of \S58). Inspired by the work of Hadamard
and Birkhoff \cite{MR1504693} Morse \cite{MR1506428} introduced
a coding of geodesics essentially corresponding to what is now known
as ,,cutting sequences{}`` and used this coding to show the existence
of a certain type of recurrent geodesics \cite{MR1501161}. 

Further ergodic properties of the geodesic flow on surfaces of constant
negative curvature given by Fuchsian groups were shown by e.g. Artin
\cite{50.0677.11}, Nielsen \cite{51.0549.01}, Koebe \cite{55.0958.03},
L\"obell \cite{MR1545081}, Myrberg \cite{MR1555338}, Hedlund \cite{0006.36503,MR1503197,MR1503297,MR1503464},
Morse and Hedlund \cite{MR1507944} and Hopf \cite{MR1501848,MR0001464}.
In this sequence of papers one can see the subject of symbolic dynamics
emerging. For a more up-to-date account of the ergodic properties
of the geodesic flow on a surface of constant negative curvature formulated
in a modern language see e.g. the introduction in Series \cite{MR594628}.

Artin's \cite{50.0677.11} approach was novel in that he used continued
fractions to code geodesics on the modular surface. After Artin, coding
and symbolic dynamics on the modular surface have been studied by
e.g. Adler and Flatto \cite{MR670077,MR779707,MR737384} and Series
\cite{MR810563}. For a recent review of different aspects of coding
of geodesics on the modular surface see for example the expository
papers by Katok and Ugarcovici \cite{MR2223106,MR2265011}. 

Other important references for the theory of symbolic dynamics and
coding of the geodesic flow on hyperbolic surfaces are e.g. Adler-Flatto
\cite{MR1085823}, Bowen and Series \cite{MR556585} and Series \cite{MR594628}.

In the present paper we study the geodesic flow on a family of hyperbolic
surfaces with one cusp and two marked points, the so-called Hecke
triangle surfaces, generalizing the modular surface. Symbolic dynamics
for a related billiard has also been studied by Fried \cite{MR1400315}.
We now give a summary of the paper. Sections 1 and 2 contain preliminary
facts about hyperbolic geometry and geodesic flows. In Section 3 we
develop the theory of $\lambda$-fractions connected to the coding
of the geodesic flow on the Hecke triangle surfaces. The explicit
discretization of the geodesic flow in terms of a Poincaré section
and Poincaré map is developed in Section 4. As an immediate application
we derive invariant measures for certain interval maps in Section
5. Some rather technical lemmas are confined to the end in Section
6.

\subsection{\label{sub:Hyperbolic-geometry-and}Hyperbolic geometry and Hecke
triangle surfaces}

Recall that any hyperbolic surface of constant negative curvature
$-1$ is given as a quotient (orbifold) $\mathcal{M}=\H/\Gamma$.
Here $\H=\left\{ z=x+iy\,|\, y>0,\, x\in\mathbb{R}\right\} $ together
with the metric $ds=\frac{\left|dz\right|}{y}$ is the \emph{hyperbolic
upper half-plane} and $\Gamma\subseteq\PSLR\cong\SLR/\left\{ \pm I_{2}\right\} $
is a Fuchsian group. Here $\SLR$ is the group of real two-by-two
matrices with determinant $1$, $I_{2}=\left(\begin{smallmatrix}1 & 0\\
0 & 1\end{smallmatrix}\right)$ and $\PSLR$ is the group of orientation preserving isometries of
$\H$. The boundary of $\H$ is $\partial\H=\mathbb{R}^{*}=\mathbb{R}\cup\left\{ \infty\right\} $.
If $g=\left(\begin{smallmatrix}a & b\\
c & d\end{smallmatrix}\right)\in\PSLR$ then $gz=\frac{az+b}{cz+d}\in\H$ for $z\in\H$ and we say that $g$
is \emph{elliptic}, \emph{hyperbolic} or \emph{parabolic} depending
on whether $\left|\Tr\, g\right|=\left|a+d\right|<2,$ $>2$ or $=2$.
The same notation applies for fixed points of $g$. In the following
we identify the elements $g\in\PSLR$ with the map it defines on $\H$.
Note that the type of fixed point is preserved under conjugation $z\rightarrow AzA^{-1}$
by $A\in\PSLR$. A parabolic fixed point is a degenerate fixed point,
belongs to $\partial\H$ and is usually called a cusp. Elliptic points
$z$ appear in pairs, one belongs to $\H$ and the other one is in
the lower half-plane $\overline{\H}$ and its stabilizer subgroup
$\Gamma_{z}$ in $\Gamma$ is cyclic of finite order $m$. Hyperbolic
fixed points appear also in pairs with $x,x^{*}\in\partial\H$, where
$x^{*}$ is said to be the conjugate point of $x$. A geodesics $\gamma$
on $\H$ is either a half-circle orthogonal to $\mathbb{R}$ or a
line parallel to the imaginary axis and the endpoints of $\gamma$
are denoted by $\gamma_{\pm}\in\partial\H$. We identify the set of
geodesics on $\H$ with $\mathcal{G}=\left\{ \left(\xi,\eta\right)\,|\,\xi\ne\eta\in\mathbb{R}^{*}\right\} $
and use $\gamma\left(\xi,\eta\right)$ to denote the oriented geodesic
on $\H$ with $\gamma_{+}=\xi$ and $\gamma_{-}=\eta$. Unless otherwise
stated all geodesics are assumed to be parametrized with hyperbolic
arc length with $\gamma\left(0\right)$ either at height $1$ if $\gamma$
is vertical or the highest point on the half-circle. It is known that
$z\in\H$ and $\theta\in\left[0,2\pi\right)\cong S^{1}$ determine
a unique geodesic (cf.~Lemma \ref{lem:z_theta->gamma_z_theta_unique_cont})
passing through $z$ whose tangent at $z$ makes an angle $\theta$
with the positive $\xi$-axis. This geodesic is denoted by $\gamma_{z,\theta}$.
It is also well known that a geodesic $\gamma\left(\xi,\eta\right)$
is closed if and only if $\xi$ and $\eta=\xi^{*}$ are conjugate
hyperbolic fixed points. 

The unit tangent bundle of $\H$, $\SH=\bigsqcup{}_{z\in\H}\left\{ \vec{v}\in T_{z}\H\,|\,\left|\vec{v}\right|=1\right\} $
is the collection of all unit vectors in the tangent planes of $\H$
with base points $z\in\H$ which we denote by $T_{z}^{1}\H$. By identifying
$\vec{v}$ with its angle $\theta$ with respect to the positive real
axis we can view $\SH$ as the collection of all pairs $\left(z,\theta\right)\in\H\times S^{1}$.
We may also view this as the set of geodesics $\gamma_{z,\theta}$
on $\H$ or equivalently as $\mathcal{G}\subseteq\mathbb{R}^{*2}$.

Let $\pi:\H\rightarrow\mathcal{M}$ be the natural projection map,
i.e. $\pi\left(z\right)=\Gamma z$ and let $\pi^{*}:\SH\rightarrow\SM$
be the extension of $\pi$ to $\SH$. Then $\gamma^{*}=\pi\gamma$
is a closed geodesic on $\mathcal{M}$ if and only if $\gamma_{+}$
and $\gamma_{-}$ are fixed points of the same hyperbolic map $g_{\gamma}\in\Gamma$.
For an introduction to hyperbolic geometry and Fuchsian groups see
e.g. \cite{katok,MR0164033,ratcliffe:Foundations}. 

\begin{defn}
\label{def:G_q_and_fundamental_domain}For an integer $q\ge3$ the
\emph{Hecke triangle group} $G_{q}\subseteq\PSLR$ is the group generated
by the maps $S:z\mapsto-\frac{1}{z}$ and $T:z\mapsto z+\lambda$
where $\lambda=\lambda_{q}=2\cos\left(\frac{\pi}{q}\right)\in\left[1,2\right)$.
The corresponding orbifold (Riemann surface) is $\mathcal{M}_{q}=G_{q}\backslash\H$,
which we sometimes identify with the standard fundamental domain of
$G_{q}$ \[
\mathcal{F}_{q}=\left\{ z\in\H\,\large{|}\,\left|\Re z\right|\le\lambda/2,\,\left|z\right|\ge1\right\} \]
with sides pairwise identified. Let $\rho=\rho_{+}=e^{\frac{\pi i}{q}}$
and $\rho_{-}=-\overline{\rho}$. We define the following oriented
boundary components of $\mathcal{F}_{q}$: $L_{0}$ is the circular
arc from $\rho_{-}$ to $\rho_{+}$. $L_{1}$ is the vertical line
from $\rho_{+}$ to $i\infty$ and $L_{-1}$ is the vertical line
from $i\infty$ to $\rho_{-}$. Thus $\partial\mathcal{F}_{q}=L_{-1}\cup L_{0}\cup L_{1}$
is the positively oriented boundary of $\mathcal{F}_{q}$. 
\end{defn}
\begin{rem}
\label{rem:relations_in_Gq}The group $G_{q}$ is a realization of
the Schwarz triangle group $\left(\frac{\pi}{\infty},\frac{\pi}{q},\frac{\pi}{2}\right)$
and it is not hard to show (see e.g. \cite[VII]{MR0164033}) that
$G_{q}$ for $q\ge3$ is a co-finite Fuchsian group with fundamental
domain $\mathcal{F}_{q}$ and the only relations \begin{equation}
S^{2}=\left(ST\right)^{q}=Id\text{ -- the identity in }\PSLR.\label{eq:G_q_relations}\end{equation}
Hence $G_{q}$ has one cusp, that is the equivalence class of parabolic
points, and two elliptic equivalence classes of orders $2$ and $q$
respectively. Note that $G_{3}=\PSLZ$ --the modular group and $G_{4}$,
$G_{6}$ are conjugate to congruence subgroups of the modular group.
For $q\ne3,4,6$ the group $G_{q}$ is non-arithmetic (cf. \cite[pp.\ 151-152]{katok}),
but in the terminology of \cite{MR1075639,MR1745404} it is semi-arithmetic,
meaning that it is possible to embed $G_{q}$ as a subgroup of a Hilbert
modular group. 
\end{rem}

\section{The Geodesic Flow on $\SM$\label{sec:The-Geodesic-Flow}}

We briefly recall the notion of the geodesic flow on a Riemann surface
$\mathcal{M}=\Gamma\backslash\H$ with $\Gamma\subset\PSLR$ a Fuchsian
group. To any $\left(z,\theta\right)\in\SH\cong\H\times S^{1}$ we
can associate a unique geodesic $\gamma=\gamma_{z,\theta}$ on $\H$
such that $\gamma\left(0\right)=z$ and $\dot{\gamma}\left(0\right)=e^{i\theta}$.
The geodesic flow on $\SH$ can then be viewed as a map $\Phi_{t}:\SH\rightarrow\SH$
with $\Phi_{t}\left(\gamma_{z,\theta}\right)=\Phi_{t}\left(z,\theta\right)=\left(\gamma_{z,\theta}\left(t\right),\dot{\gamma}_{z,\theta}\left(t\right)\right),$
$t\in\mathbb{R}$ satisfying $\Phi_{t+s}=\Phi_{t}\circ\Phi_{s}$.
The geodesic flow $\Phi^{*}$ on $\SM$ is then given by the projection
$\Phi_{t}^{*}=\pi^{*}\left(\Phi_{t}\right)$.

A more abstract and general description of the geodesic flow, which
can be extended to other homogeneous spaces, is obtained by the identification
$\SH\cong\PSLR$. Under this representation the geodesic flow corresponds
to right multiplication by the matrix $a_{t}^{-1}=\left(\begin{smallmatrix}e^{t/2} & 0\\
0 & e^{-t/2}\end{smallmatrix}\right)$ in $\PSLR$ (cf. e.g. \cite[Ch. 13]{einseidler}). 

\begin{defn}
\label{def:poincare_section}Let $\Upsilon$ be a set of geodesics
on $\H$. A hypersurface $\Sigma\subseteq\SH$ is said to be a \emph{Poincaré
section} or cross section for the geodesic flow on $\SH$ for $\Upsilon$
if any $\gamma\in\Upsilon$ intersects $\Sigma$
\begin{lyxlist}{00.00.0000}
\item [{(P1)}] transversally i.e. non-tangentially, and
\item [{(P2)}] infinitely often, i.e. $\Phi_{t_{j}}\left(\gamma\right)\in\Sigma$
for an infinite sequence of $t_{j}\rightarrow\pm\infty$. 
\end{lyxlist}
The corresponding \emph{first return map} is the map $\mathcal{T}:\Sigma\rightarrow\Sigma$
such that $\mathcal{T}\left(z,\theta\right)=\Phi_{t_{0}}\left(z,\theta\right)\in\Sigma$
and $\Phi_{t}\left(z,\theta\right)\notin\Sigma$ for $0<t<t_{0}$.
Here $t_{0}=t_{0}\left(z,\theta\right)>0$ is called the \emph{first
return time}. 

\end{defn}
Poincaré sections were first introduced by Poincaré \cite{14.0666.01}
to show the stability of periodic orbits. For examples of cross section
maps in connection with the geodesic flow on hyperbolic surfaces see
e.g.~\cite{MR1085823,MR670077}. 

The previous definition extend naturally to $\SM$ with $\Upsilon$
and $\Sigma$ replaced by $\Upsilon^{*}=\pi\left(\Upsilon\right)$
and $\Sigma^{*}=\pi^{*}\left(\Sigma\right)$. The first return map
$\mathcal{T}$ is used to obtain a discretization of the geodesic
flow, e.g. we replace $\Phi_{t}\left(z,\theta\right)$ by $\left\{ \Phi_{t_{k}}\left(z,\theta\right)\right\} $
where $t_{k}\left(z,\theta\right)$ is a sequence of consecutive first
returns. Incidentally this provides a reduction of the dynamics from
three to two dimensions and it turns out that in our example the first
return map also has a factor map, which allows us to study the three
dimensional geodesic flow with the help of an interval map (see Sections
\ref{sub:First-return-map} and \ref{sec:Applications}).

\section{$\lambda$-Continued Fraction Expansions}

\subsection{Basic concepts }

Continued fraction expansions connected to the groups $G_{q}$, the
so-called $\lambda$-fractions, were first introduced by Rosen \cite{MR0065632}
and subsequently studied by Rosen and others, cf. e.g. \cite{MR1190349,MR1853543,MR1219337}.
For the purposes of natural extensions (cf.~Section \ref{sub:Symbolic-dynamics})
the results of Burton, Kraaikamp and Schmidt \cite{MR1650073} are
analogous to ours and we occasionally refer to their results. Our
definition of $\lambda$-fractions is equivalent to Rosens definition
(cf. e.g.~\cite[\S 2]{MR0065632}).

To a sequence of integers, $a_{0}\in\mathbb{Z}$ and $a_{j}\in\mathbb{Z}^{*}=\mathbb{Z}\backslash\left\{ 0\right\} ,$
$j\ge1$ (finite or infinite) we associate a \emph{$\lambda$-fraction}
$\underline{x}=\FS{a_{0};a_{1},a_{2},\ldots}$. This $\lambda$-fraction
is identified with the point \[
x=a_{0}\lambda-\frac{1}{a_{1}\lambda-\frac{1}{a_{2}\lambda-\ddots}}=\lim_{n\rightarrow\infty}T^{a_{0}}ST^{a_{1}}\,\cdots\, ST^{a_{n}}\left(0\right)\]
if the right hand side is convergent. When there is no risk of confusion,
we sometimes write $x=\underline{x}$. The \emph{tail} of $\underline{x}$
is defined as $\FS{a_{m+1},a_{m+2},\ldots}$ for any $m\ge1$. Note
that $\underline{-x}=\FS{-a_{0};-a_{1},-a_{2},\ldots}$. If $a_{0}=0$,
we usually omit the leading $\FS{0;}$. Repetitions in a sequence
is denoted by a power, e.g. $\FS{a,a,a}=\FS{a^{3}}$ and an infinite
repetition is denoted by an overline, e.g. $\FS{a_{1},\ldots a_{k},a_{1},\ldots,a_{k},\ldots}=\FS{\overline{a_{1},\ldots,a_{k}}}$.
Such a $\lambda$-fraction is said to be \emph{periodic }with period
$k$, an \emph{eventually periodic} $\lambda$-fraction has a periodic
tail. Two $\lambda$-fractions $\underline{x}$ and $\underline{y}$
are said to be \emph{equivalent} if they have the same tail. In this
case it is easy to see that, if the fractions are convergent, then
$x=Ay$ for some $A\in G_{q}$.

The sole purpose for introducing $\lambda$-fractions is to code geodesics
by identifying the $\lambda$-fractions of their endpoints with elements
of $\mathbb{Z}^{\mathbb{N}}$. For reasons that will be clear later
(Section \ref{sub:Symbolic-dynamics}), we have to consider also bi-infinite
sequences $\mathbb{Z^{\mathbb{Z}}}$ and view $\mathbb{Z^{\mathbb{N}}}$
as embedded in $\mathbb{Z}^{\mathbb{Z}}$ with a zero-sequence to
the left. On $\mathbb{Z^{\mathbb{Z}}}$ we always use the metric $\tilde{h}$
defined by $\tilde{h}\left(\left\{ a_{i}\right\} _{i=-\infty}^{\infty},\left\{ b_{i}\right\} _{i=-\infty}^{\infty}\right)=\frac{1}{1+n}$
where $a_{i}=b_{i}$ for $\left|i\right|<n$ and $a_{n}\ne b_{n}$
or $a_{-n}\ne b_{-n}$. In this metric $\mathbb{Z^{\mathbb{Z}}}$
and $\mathbb{Z^{\mathbb{N}}}$ have the topological structure of a
Cantor set and the left- and right shift maps $\sigma^{\pm}:\mathbb{Z^{\mathbb{Z}}\rightarrow\mathbb{Z}^{\mathbb{Z}}},$
$\sigma^{\pm}\left\{ a_{j}\right\} =\left\{ a_{j\pm1}\right\} $ are
continuous. We also set $\sigma^{+}\FS{a_{1},a_{2},\ldots}=\FS{a_{2},a_{3},\ldots}$.

\subsection{Regular $\lambda$-fractions}

In the set of all $\lambda$-fractions we choose a ,,good{}`` subset,
in which almost all $x\in\mathbb{R}$ have unique $\lambda$-fractions
and in which infinite $\lambda$-fractions are convergent. The first
step is to choose a ,,fundamental region{}`` $I_{q}$ for the action
of $T:\mathbb{R}\rightarrow\mathbb{R}$, namely $I_{q}=\left[-\frac{\lambda}{2},\frac{\lambda}{2}\right]$.
Then it is possible to express one property of our ,,good{}`` subset
as follows: If in the fraction $\underline{x}=\FS{a_{0};a_{1},a_{2},\ldots}$
the first entry $a_{0}=0$, then $x\in I_{q}$. That means, we do
not allow sequences with $a_{0}=0$ correspond to points outside $I_{q}$. 

A shift-invariant extension of this property leads to the following
definition of regular $\lambda$-fractions:

\begin{defn}
\label{def:regular-nakada-exp}Let $\underline{x}=\FS{a_{0};a_{1},a_{2},\ldots}$
be a finite or infinite convergent $\lambda$-fraction and let $\underline{x}_{j}=\sigma^{j}\underline{x}=\FS{0;a_{j},a_{j+1},\ldots},\, j\ge1,$
be the $j$-th shift of $\underline{x}$. Let $x_{j}$ be the corresponding
point. Then $\underline{x}$ is said to be a \emph{regular $\lambda$-fraction}
if and only if \[
x_{j}\in I_{q},\,\text{for all}\,\, j\ge1.\tag{*}\]
A regular $\lambda$-fraction is denoted by $\RS{a_{0};a_{1},\ldots}$,
the space of all regular $\lambda$-fractions is denoted by $\RNCF$
and the subspace of \emph{infinite }regular $\lambda$-fractions with
$a_{0}=0$ is denoted by $\ZRNCF$. 
\end{defn}
For a finite fraction $\underline{x}=\FS{a_{0};a_{1},\ldots,a_{n}}$
we get $\underline{x}_{j}=\FS{0;}$ and $x_{j}=0\in I_{q}$ for $j>n$.

We will see later that regular $\lambda$-fractions can be regarded
as \emph{nearest $\lambda$-multiple continued fractions. }In the
case $q=3$ or $\lambda=1,$ nearest integer continued fractions were
studied already by Hurwitz \cite{hurwitz} in 1889. An account of
Hurwitz reduction theory can be found in Fried \cite{MR2114673} (cf.
also the H-expansions in \cite{MR2223106,MR2265011}). For general
$q$ this particular formulation of Rosens fractions was studied by
Nakada \cite{MR1402490}. 

For the remainder of the paper we let $h=\frac{q-3}{2}$ if $q$ is
odd and $h=\frac{q-2}{2}$ if $q$ is even. The following Lemma is
an immediate consequence of \cite[(4)]{MR1650073}. 

\begin{lem}
\label{lem:lambda_half_finite_expansion}The points $\mp\frac{\lambda}{2}$
have finite regular $\lambda$-fractions given by \begin{align*}
\mp\frac{\lambda}{2}= & \begin{cases}
\RS{\left(\pm1\right)^{h}}, & \mbox{for }q\,\mbox{even,}\\
\RS{\left(\pm1\right)^{h},\pm2,\left(\pm1\right)^{h}}, & \mbox{for }q\,\mbox{odd}.\end{cases}\end{align*}

\end{lem}

\begin{lem}
\label{lem:expansion_of_1_odd_q}If $q$ is odd, the point $x=1$
has the finite regular $\lambda$-fraction \[
1=\RS{1;1^{h}}.\]

\end{lem}
\begin{proof}
Since $a=\RS{1;1^{h}}=T\left(ST\right)^{h}\left(0\right),$ one has
also $-a=T^{-1}\left(ST^{-1}\right)^{h}\left(0\right)$. From identity
(\ref{eq:G_q_relations}) we get \[
Sa=\left(ST\right)^{h+1}\left(0\right)=\left(T^{-1}S\right)^{h+2}\left(0\right)=T^{-1}\left(ST^{-1}\right)^{h}ST^{-1}S\left(0\right)=T^{-1}\left(ST^{-1}\right)^{h}\left(0\right),\]
and hence $-1/a=-a$. Since $a>0$, this implies that $a=1$. 
\end{proof}
\begin{defn}
\label{def:Fq}Let $\left\lfloor x\right\rfloor $ be the floor function
defined by $\left\lfloor x\right\rfloor =n\Leftrightarrow n<x\le n+1$
for $x>0$, respectively $n\le x<n+1$ for $x\le0$, and let $\nlm{x}=\left\lfloor \frac{x}{\lambda}+\frac{1}{2}\right\rfloor $
be the corresponding nearest $\lambda$-multiple function. Then define
$\Fq:I_{q}\rightarrow I_{q}$ by\[
\Fq x=\begin{cases}
-\frac{1}{x}-\nlm{-\frac{1}{x}}\lambda, & x\in I_{q}\backslash\left\{ 0\right\} ,\\
0, & x=0.\end{cases}\]

\end{defn}
\begin{lem}
\label{lem:generating_map_Fq_gives_regular_expansion}For $x\in\mathbb{R}$
the following algorithm gives a finite or infinite regular $\lambda$-fraction
$\Cq\left(x\right)=\RS{a_{0};a_{1},\ldots}$ corresponding to $x$: 
\begin{lyxlist}{00.00.0000}
\item [{(i)}] Set $a_{0}:=\nlm{x}$ and $x_{1}=x-a_{0}\lambda$. 
\item [{(ii)}] Set $x_{j+1}:=\Fq x_{j}=-\frac{1}{x_{j}}-a_{j}\lambda$,
$j\ge1$, with $a_{j}=\nlm{\frac{-1}{x_{j}}},\, j\ge1.$ 
\end{lyxlist}
If $x_{j}=0$ for some $j$, the algorithm stops and gives a finite
regular $\lambda$-fraction. 

\end{lem}
\begin{proof}
By definition we see that $x_{j+1}=T^{-a_{j}}Sx_{j}$, $j\ge1$, and
hence $x=T^{a_{0}}ST^{a_{1}}\,\cdots\, ST^{a_{n}}x_{n}$ for any $n\ge1$.
If $\underline{x}=\RS{a_{0};a_{1},.\ldots},$ then for $j\ge1$ $\underline{x}_{j}=\sigma^{j}\underline{x}=\RS{0;a_{j},\ldots}$
corresponds to the point $x_{j}$ and condition ({*}) of Definition
\ref{def:regular-nakada-exp} is fulfilled, since $\Fq$ maps $I_{q}$
to itself and $x_{1}\in I_{q}$. 
\end{proof}
\begin{rem}
We say that $\Fq$ is a \emph{generating map} for the regular $\lambda$-fractions.
It is also clear from Lemma \ref{lem:generating_map_Fq_gives_regular_expansion}
that $\Fq$ acts as a shift map on the space $\ZRNCF$, i.e. $\Cq\left(\Fq x\right)=\sigma\Cq\left(x\right)$.
\end{rem}
An immediate consequence of Lemma \ref{lem:generating_map_Fq_gives_regular_expansion}
is the following corollary: 

\begin{cor}
\label{cor:infinite-regular-are-unique}If $x$ has an infinite regular
$\lambda$-fraction, then it is unique and equal to $\Cq\left(x\right)$
as given by Lemma \ref{lem:generating_map_Fq_gives_regular_expansion}.
\end{cor}

The above choice of floor function implies that $\Fq$ is an odd function
and that $\nlm{\pm\frac{\lambda}{2}}=0$ in agreement with Lemma \ref{lem:lambda_half_finite_expansion}.
The ambiguity connected to the choice of floor function at integers
affects only the points $x=\frac{2}{\lambda\left(1-2k\right)}$ where
$\frac{-1}{x\lambda}+\frac{1}{2}=k\in\mathbb{Z}$ and $\Fq x=\left(k-\left\lfloor k\right\rfloor \right)\lambda-\frac{\lambda}{2}\in=\pm\frac{\lambda}{2}$.
By Lemma \ref{lem:lambda_half_finite_expansion} we conclude, that
any point, which has more than one regular $\lambda$-fraction, is
$\Fq$-equivalent to $\pm\frac{\lambda}{2}$ and hence has a finite
$\lambda$-fraction. 

We can produce in this way a regular $\lambda$-fraction $\Cq\left(x\right)$
as a code for any $x\in\mathbb{R}$. For the purpose of symbolic dynamics
we prefer to have an intrinsic description of the members of the space
$\ZRNCF$ formulated in terms of so-called \emph{forbidden blocks},
i.e. certain subsequences which are not allowed. From Definition \ref{def:regular-nakada-exp}
it is clear, which subsequences are forbidden and how to remove them
by rewriting the sequence, using the fraction with a leading $a_{0}\ne0$
for any point outside $I_{q}$ instead of $a_{0}=0$. 

\begin{lem}
\label{lem:forbidden-block-even}Let $q$ be even and set $h=\frac{q-2}{2}$.
Since by Lemma \ref{lem:lambda_half_finite_expansion} one has $\mp\frac{\lambda}{2}=\RS{\left(\pm1\right)^{h}}$
the following blocks are forbidden: $\FS{\left(\pm1\right)^{h},\pm m}$
with $m\ge1$. Using $\left(ST\right)^{2h+2}=1$ such blocks can be
rewritten as \begin{align*}
\FS{a,\left(\pm1\right)^{h},\pm m,b} & \rightarrow\begin{cases}
\RS{a\mp1,\left(\mp1\right)^{h},\pm m\mp1,b}, & m\ge2,\\
\RS{a\mp1,\left(\mp1\right)^{h-1},b\mp1}, & m=1,\end{cases}a,b\in\mathbb{Z}.\end{align*}

\end{lem}

\begin{lem}
\label{lem:forbidden-block-odd}Let $q$ be odd and set $h=\frac{q-3}{2}$.
Since by Lemma \ref{lem:lambda_half_finite_expansion} one knows that
$\mp\frac{\lambda}{2}=\RS{\left(\pm1\right)^{h},\pm2,\left(\pm1\right)^{h}}$
the following two different types of blocks are forbidden: $\FS{\left(\pm1\right)^{h+1}}$
and $\FS{\left(\pm1\right)^{h},\pm2,\left(\pm1\right)^{h},\pm m},$
with $m\ge1$. Using $\left(ST\right)^{2h+3}=1$ for $q\ge5$ these
blocks can be rewritten as \begin{align*}
\FS{a,\left(\pm1\right)^{h+1},b} & \rightarrow\FS{a\mp1,\left(\mp1\right)^{h},b\mp1},\\
\FS{a,\left(\pm1\right)^{h},\pm2,\left(\pm1\right)^{h},\pm m,b} & \rightarrow\begin{cases}
\FS{a\mp1,\left(\mp1\right)^{h},\mp2,\left(\mp1\right)^{h},\pm m\mp1,b}, & m\ge2,\\
\FS{a\mp1,\left(\mp1\right)^{h},\mp2,\left(\mp1\right)^{h-1},b\mp1}, & m=1\end{cases}\end{align*}
and for $q=3$ as \begin{align*}
\FS{a,\pm1,b} & \rightarrow\FS{a\mp1,b\mp1},\\
\FS{a,\pm2,\pm m,b} & \rightarrow\begin{cases}
\FS{a\mp1,\mp2,\pm m\mp1,b}, & m\ge2,\\
\FS{a\mp1,b\mp2,}, & m=1.\end{cases}\end{align*}

\end{lem}
\begin{rem}
\label{rem:connection_with_Rosen_fractions}It is easy to see that
Rosens $\lambda$-fractions \cite{MR0065632} can be expressed as
words in the generators $T,$ $S$ and $JS$ of the group $G_{q}^{*}=<G_{q},J>\subseteq\PGLR,$
where $J:z\mapsto-\overline{z}$ is the reflection in the imaginary
axis. Since $J$ is an involution of $G_{q}$, e.g. $JTJ=T^{-1}$
and $JSJ=S$, it is easy to see that Rosen's and our notions of $\lambda$-fractions
are equivalent: e.g. in the $G_{q}$-word identified with our $\lambda$-fraction
we replace any $T^{-a}$ by $JT^{a}J$, $a\ge1$. Algorithmically
this means for a $\lambda$-fraction with entries $a_{j}$ that the
corresponding Rosen fraction has entries $(\epsilon_{j},|a_{j}|)$
where $\epsilon_{1}=-\sign\left(a_{1}\right)$ and $\epsilon_{j}=-\sgn(a_{j-1}a_{j})$
for $j\ge2$.

From the definition of regular $\lambda$-fractions it is clear that
Rosen's reduced $\lambda$-fractions \cite[Def.\ 1]{MR0065632} correspond
to a fundamental interval $\RS{0,\frac{\lambda}{2}}$ for the action
of the group $\left\langle T,J\right\rangle $ together with the choices
made for finite fractions in \cite[Def.\ 1 (4)-(5)]{MR0065632}. It
is easy to verify, for example using the forbidden blocks, that a
finite fraction not equivalent to $\pm\frac{\lambda}{2}$ or an infinite
regular $\lambda$-fraction correspond to a reduced $\lambda$-fraction
of Rosen. The main difference between our regular and Rosens reduced
$\lambda$-fractions is that any $\lambda$-fraction equivalent to
$\pm\frac{\lambda}{2}$ has two valid regular $\lambda$-fractions.
The root of this non-uniqueness is our choice of a closed interval
$I_{q}$ which is in turn motivated by our Markov partitions in Section
\ref{sub:Markov-Partitions-for}. 

It is then clear, that those results of \cite{MR0065632} and \cite{MR1650073}
pertaining to infinite reduced $\lambda$-fractions can be applied
directly to our regular $\lambda$-fractions. 
\end{rem}
\begin{lem}
\label{lem:An-infinite-regular-converges}An infinite $\lambda$-fraction
without forbidden blocks is convergent.
\end{lem}
\begin{proof}
This follows from \cite[Thm.\ 5]{MR0065632} and Remark \ref{rem:connection_with_Rosen_fractions}.
For more details see also \cite{cf_main}.
\end{proof}
An immediate consequence of Definition \ref{def:regular-nakada-exp}
and Lemmas \ref{lem:forbidden-block-even}, \ref{lem:forbidden-block-odd}
and \ref{lem:An-infinite-regular-converges} is the following 

\begin{cor}
A $\lambda$-fraction is regular if and only if it does not contain
any forbidden block. 
\end{cor}
It is important that rewriting forbidden blocks can be done sequentially.

\begin{lem}
\label{lem:rewrite_not_affect_left}If the $\lambda$-fraction $\FS{a_{0};a_{1},\ldots}$
has a forbidden block beginning at $a_{n}$, then the subsequence
$\FS{a_{0};a_{1},\ldots,a_{n-2}}$ is not affected by rewriting this
forbidden block.
\end{lem}
\begin{proof}
This is easy to verify in the different cases arising in Lemmas \ref{lem:forbidden-block-even}
and \ref{lem:forbidden-block-odd} by simply noting that $a_{n}=\pm1$
so $a_{n-1}\ne\pm1$. For the complete details see \cite{cf_main}.
\end{proof}

\subsection{Dual regular $\lambda$-fractions}

To encode the orbits of the geodesic flow in terms of a discrete invertible
dynamical system it turns out that we still need another kind of $\lambda$-fraction,
the so-called \emph{dual regular} $\lambda$-fraction. In the case
$q=3$ this was already introduced by Hurwitz \cite{hurwitz}, see
also \cite[p.\ 15]{MR2265011}. 

Consider the set of $\lambda$-fractions $\underline{y}=\FS{0;b_{1},\ldots}$
which do not contain any reversed forbidden block, i.e. a forbidden
block given in Definitions \ref{lem:forbidden-block-even} or \ref{lem:forbidden-block-odd}
read in reversed order. 

Let $R$ be the largest number in this set and define $r=R-\lambda$
and $I_{R}=\left[-R,R\right]$. 

\begin{lem}
\label{lem:R-is-given-by-lambda-fraction}The number $R$ is given
by the following regular $\lambda$-fraction \[
R=\begin{cases}
\RS{1;\overline{1^{h-1},2}}, & \mbox{for }q\,\mbox{even},\\
\RS{1;\overline{1^{h},2,1^{h-1},2}}, & \text{for }q\ge5\,\mbox{odd},\\
\RS{1;\overline{3}}, & \mbox{for }q=3.\end{cases}\]

\end{lem}
\begin{proof}
To obtain the largest number we use the lexicographic ordering on
$\ZRNCF$ implied by the following sequence of easily verifiable inequalities:
\[
\RS{1}<\RS{2}<\cdots<0<\cdots<\RS{-2}<\RS{-1}.\]
For the details on how this ordering extends to infinite sequences
see \cite{cf_main}. Thus we want to take as many digits as possible
close to $-1$ in the $\lambda$-fraction of $R$. The constraints
set by the reversed forbidden blocks in Lemmas \ref{lem:forbidden-block-even}
and \ref{lem:forbidden-block-odd} clearly give the following expression
for $R$: $R=\FS{0;\left(-1\right)^{h},\overline{-2,\left(-1\right)^{h-1}}}$
for even $q$, $R=\FS{0;\left(-1\right)^{h},\overline{-2,\left(-1\right)^{h},-2,\left(-1\right)^{h-1}}}$
for odd $q\ge5$ and $R=\FS{0;-2,\overline{-3}}$ for $q=3$. The
Lemma is then clear by rewriting these numbers recursively into regular
$\lambda$-fractions using Lemmas \ref{lem:forbidden-block-even}
and \ref{lem:forbidden-block-odd}.
\end{proof}
\begin{lem}
\label{lem:properties_of_R_even}For even $q$ we have the identity
$R=1.$ 
\end{lem}
\begin{proof}
Consider the action of $S$ on $R$: $SR=\RS{0;1^{h},\overline{2,1^{h-1}}}=\RS{-1;\overline{\left(-1\right)^{h-1},-2}}=-R$.
Hence $-1/R=-R$ and since $R>0$ we must have  $R=1$. 
\end{proof}
\begin{lem}
\label{lem:properties_of_R_odd}For odd $q$ we have $\frac{\lambda}{2}<R<1$
and 
\end{lem}
\begin{lyxlist}{00.00.0000}
\item [{a)}] $-R=\left(TS\right)^{h+1}R$,
\item [{b)}] $R^{2}+\left(2-\lambda\right)R-1=0$
\end{lyxlist}
\begin{proof}
From the explicit expansions of $\frac{\lambda}{2}$ and $1$ in Lemmas
\ref{lem:lambda_half_finite_expansion} and \ref{lem:expansion_of_1_odd_q}
together with the lexicographic ordering mentioned in the proof of
Lemma \ref{lem:R-is-given-by-lambda-fraction} it is clear that $\frac{\lambda}{2}<R<1$
(or see the proof of Lemma 3.3 in \cite{MR1650073}). By rewriting
as in Lemma \ref{lem:forbidden-block-odd} for $q\ge5$ we get 

$SR=\FS{0;1^{h+1},\overline{2,1^{h-1},2,1^{h}}}=\RS{-1;\left(-1\right)^{h-1},-2,\overline{\left(-1\right)^{h},-2,\left(-1\right)^{h-1},-2}}$
and deduce that $R=ST^{-1}\left(ST^{-1}\right)^{h-1}ST^{-2}T\left(-R\right)=\left(ST^{-1}\right)^{h+1}\left(-R\right)$
and hence $-R=\left(TS\right)^{h+1}R$, which is identity a). A similar
rewriting works for $q=3$. Using the following explicit formula for
the matrix $\left(TS\right)^{n}$ (cf. e.g. \cite[p.\ 1279]{MR1650073})
\begin{equation}
\left(TS\right)^{n}=\frac{1}{\sin^{2}\frac{\pi}{q}}\left(\begin{smallmatrix}B_{n+1} & -B_{n}\\
B_{n} & -B_{n-1}\end{smallmatrix}\right),\,\mbox{where}\, B_{n}=\sin\frac{n\pi}{q}\label{eq:TS-explicit}\end{equation}
and some elementary trigonometry gives $\left(TS\right)^{h+1}R=\frac{-R+1}{-R+\lambda-1}=-R$
which implies identity b).
\end{proof}
\begin{rem}
Using the representation (\ref{eq:TS-explicit}) one can also show
that the map $A_{r}$ fixing $r=R-1$ is given by $A_{r}=\left(ST\right)^{h+1}T\left(ST\right)^{h}T=\left(\begin{smallmatrix}2-2\lambda & \lambda-2\lambda^{2}\\
\lambda & 2+\lambda^{2}\end{smallmatrix}\right)$ for odd $q$ and $A_{r}=\left(ST\right)^{h-1}ST^{2}=\frac{1}{4\sin^{2}\frac{\pi}{q}}\left(\begin{smallmatrix}2-\lambda^{2} & 7\lambda-3\lambda^{3}\\
\lambda & 2+\lambda^{2}\end{smallmatrix}\right)$ for even $q$.
\end{rem}
\begin{defn}
\label{def:dual_regular_lambda-fraction}Let $\underline{y}=\FS{b_{0};b_{1},\ldots}$
be a finite or infinite $\lambda$-fraction. Set $\underline{y}_{0}=\FS{0;b_{1},\ldots}$,
$\underline{y}_{j}=\sigma^{j-1}\underline{y}_{0}=\FS{0;b_{j},\ldots}$
and $y_{j}$, $j\ge0$, the corresponding point in $\mathbb{R}$.
Then $\underline{y}$ is said to be a \emph{dual regular $\lambda$-fraction
}if and only if it has the following properties:\begin{align}
\mbox{if}\,\,\, b_{0}=0 & \Rightarrow y\in I_{R},\,\tag{D1}\\
\mbox{if}\,\,\, b_{0}\ne0 & \Rightarrow y_{0}\in\sign\left(b_{0}\right)\,\left[r,R\right],\,\text{and\tag{D2}}\\
y_{j+1} & \in\sign\left(-y_{j}\right)\left[r,R\right]\,\text{for all }\, j\ge1.\tag{D3}\end{align}
A dual regular $\lambda$-fraction is denoted by $\DS{b_{0};b_{1},\ldots}$,
the space of all dual regular $\lambda$-fractions by $\DRNCF$ and
the subspace of all \emph{infinite} sequences in $\DRNCF$ with leading
$0$ by $\ZDRNCF$. 
\end{defn}
Uniqueness of a subset of dual regular $\lambda$-fractions is again
asserted using a generating map. 

\begin{defn}
\label{def:Fqd_dual_generating_function}Let $\left\lfloor \cdot\right\rfloor $
be the floor function from Definition \ref{def:Fq} and consider the
shifted nearest $\lambda$-multiple function $\nlms{y}=\left\lfloor \frac{y}{\lambda}+\frac{R}{\lambda}\right\rfloor $
if $y\le0$ and $\nlms{y}=\left\lfloor \frac{y}{\lambda}-\frac{r}{\lambda}\right\rfloor $
if $y>0$. For $I_{R}=\left[-R,R\right]$ we define the map $\Fqd:I_{R}\rightarrow I_{R}$
by \begin{eqnarray*}
\Fqd y & = & \begin{cases}
-\frac{1}{y}-\nlms{-\frac{1}{y}}\lambda, & y\in I_{R}\backslash\left\{ 0\right\} ,\\
0, & y=0.\end{cases}\end{eqnarray*}

\end{defn}
\begin{lem}
\label{lem:generate_dual}For $y\in\mathbb{R}$ the following algorithm
produces a finite or infinite dual regular $\lambda$-fraction $\Cd\left(y\right)=\DS{b_{0};b_{1},\ldots}$
corresponding to $y$:
\begin{lyxlist}{00.00.0000}
\item [{(i)}] Let $b_{0}=\nlms{y}$ and $y_{1}=y-b_{0}\lambda$. 
\item [{(ii)}] Set $y_{j+1}=\Fqd y_{j}=-\frac{1}{y_{j}}-b_{j}\lambda$,
i.e. $b_{j}=\nlms{-\frac{1}{y_{j}}},$ $j\ge1$. 
\end{lyxlist}
If $y_{j}=0$ for some $j$ the algorithm stops and one obtains a
finite dual regular $\lambda$-fraction. 

\end{lem}
\begin{proof}
It is easy to verify that $\nlms{y}=0$ $\Leftrightarrow$ $y\in I_{R}$
and that in general $x-\nlms{x}\in\left[r,R\right]$ for $x\ge R$
and $x-\nlms{x}\in\left[-R,-r\right]$ for $x\le-R$. It is thus clear
that (D1) and (D2) are automatically fulfilled and it follows that
$\Fqd$ maps $\left[-R,0\right]$ into $\left[r,R\right]$ and $\left[0,R\right]$
into $\left[-R,-r\right]$. Hence condition (D3) is also satisfied.
\end{proof}
\begin{rem}
\label{rem:Fqd_is_generating_map_for_dual_expansion}We say that $\Fqd$
is a \emph{generating map} for the dual regular $\lambda$-fractions
and it is easily verified that $\Fqd$ acts as a left shift map on
$\ZDRNCF$. 
\end{rem}
It is easily seen that the points affected by the choice of floor
function appearing in $\nlms{\cdot}$ (cf. $\pm\frac{\lambda}{2}$
in the regular case) are exactly those that are equivalent to $\pm r$.
Hence we obtain the following corollary. 

\begin{cor}
\label{cor:dual_regular_unique}If $y$ has an infinite dual regular
$\lambda$-fraction expansions which is not equivalent to the expansion
of $\pm r$ then it is unique and is equal to $\Cd\left(y\right)$. 
\end{cor}

\begin{lem}
A $\lambda$-fraction $\underline{y}=\FS{b_{0};b_{1},\ldots}$ is
dual regular if and only if the sequence $\underline{y}_{0}$ does
not contain any reversed forbidden blocks. Thereby $\underline{y}_{0}=\underline{y}$
if $b_{0}=0$ and $\underline{y}_{0}=S\underline{y}=\FS{0;b_{0},b_{1},\ldots}$
if $b_{0}\ne0$.
\end{lem}
\begin{proof}
Consider $q$ even and a reversed forbidden block of the form $\FS{m,1^{h}}$
with $m\ge1$. Suppose $b_{0}\ne0$. If $\underline{y}$ contains
such a forbidden block, we have $\underline{y}_{j}=\FS{m,1^{h},b_{h+j+1},\ldots}<0$
for some $j\ge0$ and therefore $\underline{y}_{j+1}=\FS{1^{h},b_{h+j+1},\ldots}<\RS{\overline{1^{h-1},2}}=r$,
i.e. $\underline{y}$ is not dual regular since $\sign\left(-y_{j}\right)=\sign\left(b_{j}\right)$
so we violate either (D2) or (D3) in Definition \ref{def:dual_regular_lambda-fraction}).
In the other direction, if $b_{j}\ge1$ and $\underline{y}_{j}=\FS{b_{j},b_{j+1},\ldots}$
for some $j\ge0$ does not contain a reversed forbidden block, then
$R\ge\underline{y}_{j+1}>\RS{\overline{1^{h-1},2}}=r$, i.e. (D2)
or (D3) is satisfied. Analogous arguments work for $b_{0}=0$, for
forbidden blocks involving $-1$s and for the case of odd $q$. For
more details see \cite{cf_main}.
\end{proof}
\begin{lem}
An infinite $\lambda$-fraction without reversed forbidden blocks
converges.
\end{lem}
\begin{proof}
The proof is similar as for the the regular $\lambda$-fractions.
For details see e.g.~\cite{cf_main}.
\end{proof}
\begin{rem}
Just as the regular $\lambda$-fractions are equivalent to the reduced
Rosen $\lambda$-fractions, one can show that the dual regular $\lambda$-fractions
are essentially equivalent to a particular instance of so-called $\alpha$-Rosen
$\lambda$-fractions, see \cite{metrical-alpha-rosen} and \cite{MR646050}
(in the case $q=3$). Note that $\nlms{y}=\left\lfloor \frac{y}{\lambda}+1-\frac{R}{\lambda}\right\rfloor $
for $y>0$. Hence $\Fqd x=T_{\alpha}\left(x\right)$ with $\alpha=\frac{R}{\lambda}$
for $x<0$ where $T_{\alpha}$ is the generating map of the $\alpha$-Rosen
fractions of \cite{metrical-alpha-rosen}. 
\end{rem}

\subsection{\label{sub:Symbolic-dynamics}Symbolic dynamics and natural extensions}

An introduction to symbolic dynamics and coding can be found in e.g.
\cite{MR1369092}. See also \cite{MR1085823,MR594628} or \cite[Appendix C]{MR1085823}.
Our underlying alphabet is infinite, $\mathcal{N}=\mathbb{Z}^{*}=\mathbb{Z}\backslash\left\{ 0\right\} $.
The dynamical system $\left(\mathcal{N}^{\mathbb{Z}_{+}},\sigma^{+}\right)$
is called the \emph{one-sided} \emph{full $\mathcal{N}-$shift. }Since
the forbidden blocks (cf. Definitions \ref{lem:forbidden-block-even}
and \ref{lem:forbidden-block-odd}) imposing the restrictions on $\ZRNCF$
and $\ZDRNCF$ all have finite length it follows that $\left(\ZRNCF,\sigma^{+}\right)$
and $\left(\ZDRNCF,\sigma^{+}\right)$ are both \emph{one-sided subshifts
of finite type} (cf. \cite[Thm.\ C7]{MR1085823}).

One can show that $\Cq:I_{q}\rightarrow\RNCF$ and $\Cd:I_{R}\rightarrow\DRNCF$
as given by Lemmas \ref{lem:generating_map_Fq_gives_regular_expansion}
and \ref{lem:generate_dual} are continuous and we call these the
regular and dual regular coding map respectively. Let $\mathbb{R}^{\infty}=\left\{ x\in\mathbb{R}\,|\,\Cd\left(x\right)\,\,\mbox{infinite}\right\} =\mathbb{R}\backslash G_{q}\left(\infty\right)$
be the set of ,,$G_{q}$-irrational points{}`` and set $I_{\alpha}^{\infty}=I_{\alpha}\cap\mathbb{R}^{\infty}$
for $\alpha=q,R$. Since the set $G_{q}\left(\infty\right)$ of cusps
of $G_{q}$ is countable it is clear, that the Lebesgue measure of
$I_{\alpha}^{\infty}$ is equal to that of $I_{\alpha}$, $\alpha=q,R$.
From \cite{cf_main} we see, that the restrictions $\Cq:I_{q}^{\infty}$$\rightarrow\ZRNCF$
and $\Cd:I_{R}^{\infty}\rightarrow\ZDRNCF$ are homeomorphisms. Since
$\sigma^{+}=\Cq\circ\Fq\circ\Cq^{-1}$ on $\ZRNCF$ and $\sigma^{+}=\Cd\circ\Fqd\circ\Cd\,^{-1}$
on $\ZDRNCF$ it follows that the one-sided subshifts $\left(\ZRNCF,\sigma^{+}\right)$
and $\left(\ZDRNCF,\sigma^{+}\right)$ are topologically conjugate
to the abstract dynamical systems $\left(\Fq,I_{q}^{\infty}\right)$
and $\left(\Fqd,I_{R}^{\infty}\right)$ respectively (see \cite[p.\ 319]{MR1085823}).

Consider the set of regular bi-infinite sequences $\BIRNCF\subset\ZDRNCF\times\ZRNCF\subset\mathbb{Z}^{\mathbb{Z}}$
consisting of precisely those $\RS{\ldots,b_{2},b_{1}\centerdot a_{1},a_{2},\ldots}$
which do not contain any forbidden block. Then $\left(\BIRNCF,\sigma\right)$
is a two-sided subshift of finite type extending the one-sided subshift
$\left(\ZRNCF,\sigma^{+}\right),$ where $\sigma=\sigma^{+}$ and
$\sigma^{-1}=\sigma^{-}$. If $\Cd\left(y\right)=\DS{b_{1},b_{2},\ldots}\in\ZDRNCF$
and $\Cq\left(x\right)=\RS{a_{1},a_{2},\ldots}\in\ZRNCF$ we define
the coding map $\bi:I_{q}\times I_{R}\rightarrow\mathbb{Z}^{2}$ by
$\bi\left(x,y\right)=\Cd\left(y\right).\Cq\left(x\right)=\RS{\ldots,b_{2},b_{1}\centerdot a_{1},a_{2},\ldots}$.
In the next section we will see that there exists a domain $\Omega\subset I_{q}\times I_{R}$
such that $\bi_{|\Omega^{\infty}}:\Omega^{\infty}\rightarrow\BIRNCF$
is one-to-one and continuous (here $\Omega^{\infty}=\Omega\cap I_{q}^{\infty}\times I_{R}^{\infty}$,
i.e. we neglect points $\left(x,y\right)$ where either $x$ or $y$
has a finite $\lambda$-fraction). The \emph{natural extension}, $\Fqx$,
of $\Fq$ to $\Omega^{\infty}$ is defined by the condition that $\left(\BIRNCF,\sigma\right)$
is topologically conjugate to $\left(\Fqx,\Omega^{\infty}\right)$,
i.e. by the relations $\sigma^{+}=\bi\circ\Fqx\circ\bi^{-1}$ and
$\sigma^{-}=\bi\circ\Fqx^{-1}\circ\bi^{-1}$, meaning that $\Fqx\left(x,y\right)=\left(\Fq x,\frac{-1}{y+a_{1}\lambda}\right)$
with $a_{1}=\nlm{\frac{-1}{x}}$ and $\Fqx^{-1}\left(x,y\right)=\left(\frac{-1}{x+b_{1}\lambda},\Fqd y\right)$
with $b_{1}=\nlms{\frac{-1}{y}}$.

\subsection{\label{sub:Markov-Partitions-for}Markov Partitions for the generating
map $\Fq$}

To construct a Markov partition of the interval $I_{q}$ with respect
to $\Fq$ we consider the orbits of the endpoints $\pm\frac{\lambda}{2}$.
For $x\in I_{q}$ or $y\in I_{R}$ we define the $\Fq$-\emph{orbit}
and $\Fqd$-\emph{orbit }of $x$ and $y$ respectively as \[
\mathcal{O}\left(x\right)=\left\{ \Fq^{j}x\,|\, j\in\mathbb{Z}_{+}\right\} ,\quad\mathcal{O}^{*}\left(y\right)=\left\{ \Fqd\,^{j}y\,|\, j\in\mathbb{Z}_{+}\right\} .\]
By Lemma \ref{lem:lambda_half_finite_expansion} it is clear that
$\mathcal{O}\left(\pm\frac{\lambda}{2}\right)$ is a finite set. Define
$\np=\#\left\{ \mathcal{O}\left(\frac{\lambda}{2}\right)\right\} -1=\frac{q-2}{2}=h$
for even $q$ and $\np=q-2=2h+1$ for odd $q$. Let $-\frac{\lambda}{2}=\phi_{0}<\phi_{1}<\cdots<\phi_{\np}=0$
be an ordering of $\mathcal{O}\left(-\frac{\lambda}{2}\right)$, set
$\mathcal{I}_{j}=\left[\phi_{j-1},\phi_{j}\right)$ and $\mathcal{I}_{-j}=-\mathcal{I}_{j}$
for $1\le j\le\np$. It is easy to verify that the closure of the
intervals form a \emph{Markov partition} of $I_{q}$ for $\Fq$. I.e.
$\left\{ \overline{\mathcal{I}}_{j}\right\} $ covers $I_{q}$, overlaps
only at endpoints and $\Fq$ maps endpoints to endpoints. Since the
alphabet $\mathcal{N}$ is infinite there exist also another Markov
partition of the form $J_{n}=\left[\frac{-2}{\lambda\left(2n-1\right)},\frac{-2}{\lambda\left(2n+1\right)}\right]\cap\left[-\frac{\lambda}{2},0\right]=-J_{-n},$
$n=1,2,\ldots$ $q>4,$ $n=2,3,\ldots$ $q=3$, where $\Fq\,_{|J_{n}}=\frac{-1}{x}-n\lambda$
and hence expanding and bijective unless $n=1$ for $q\ge3$ or $n=2$
for $q=3$. 

From the explicit formula of $\Fqx\,^{-1}$ it is clear that we also
need to consider the orbits of the endpoints of $\pm\left[r,R\right]$.
From Lemmas \ref{lem:properties_of_R_even} and \ref{lem:properties_of_R_odd}
we see that $\#\left\{ \mathcal{O}^{*}\left(-R\right)\right\} =\np+1$.
Set $r_{0}=-R$ and let $0>r_{1}>r_{2}>\cdots>r_{\np}=r>-R=r_{0}$
be an ordering of $\mathcal{O}^{*}\left(-R\right)=\left\{ r_{j}\right\} $.
One can verify that $r_{\np+1-j}\in\mathcal{I}_{j},$ $1\le j\le\np$.
Define the intervals $\mathcal{R}_{j}=\left[r_{j},R\right]=-\mathcal{R}_{-j},\,1\le j\le\np$,
the rectangles $\Omega_{j}=\mathcal{I}_{j}\times\mathcal{R}_{j},\,1\le\left|j\right|\le\np$
and finally the domain $\Omega=\cup_{\left|j\right|\le\np}\Omega_{j}$.
We also set $\Omega^{\infty}$$=\Omega\cap I_{q}^{\infty}\cap I_{R}^{\infty}$. 

\begin{rem}
\label{rem:For-even-q-definitions}For even $q$ we have $\phi_{0}=-\frac{\lambda}{2}=\RS{1^{h}},$
 $r=\RS{\overline{1^{h-1},2}}$ and $\np=h$ where $h=\frac{q-2}{2}$
(see Lemma \ref{lem:lambda_half_finite_expansion} and \ref{lem:R-is-given-by-lambda-fraction}).
It is then easy to verify that\begin{align*}
\phi_{j} & =F_{q}^{j}\left(\phi_{0}\right)=-\phi_{-j}=\RS{1^{h-j}},\,0\le j\le h,\\
r_{j} & =F_{q}^{h-j}\left(r\right)=\RS{1^{j-1},\overline{2,1^{h-1}}},\,1\le j\le h,\\
\mathcal{I}_{j} & =\left[\phi_{j-1},\phi_{j}\right)=\left[\left[1^{h+1-j}\right],\left[1^{h-j}\right]\right)=-\mathcal{I}_{-j},\,1\le j\le h,\\
\mathcal{R}_{j} & =\left[r_{j},R\right]=-\mathcal{R}_{-j},\,1\le j\le h.\end{align*}

\end{rem}

\begin{rem}
\label{rem:For-odd-q-definitions}For odd $q\ge5$ we have $\phi_{0}=-\frac{\lambda}{2}=\RS{1^{h},2,1^{h}}$,
$r=\RS{\overline{1^{h},2,1^{h-1},2}}$ and $\np=2h+1$ where $h=\frac{q-3}{2}$
(see Lemma \ref{lem:lambda_half_finite_expansion} and \ref{lem:R-is-given-by-lambda-fraction}).
It is then easy to verify that \begin{align*}
\phi_{2j} & =F_{q}^{j}\left(\phi_{0}\right)=\RS{1^{h-j},2,1^{h}},\,0\le j\le h,\\
\phi_{2j-1} & =F_{q}^{h+j}\left(\phi_{0}\right)=\RS{1^{h+1-j}},\,1\le j\le h+1,\\
r_{2j+1} & =\RS{1^{j},2,\overline{1^{h-1},2,1^{h},2}},\,0\le j\le h\quad\mbox{and}\\
r_{2j} & =\RS{1^{j-1},2,\overline{1^{h},2,1^{h-1},2}},\,1\le j\le h.\end{align*}
Hence \begin{align*}
\mathcal{I}_{2j+1} & =\left[\phi_{2j},\phi_{2j+1}\right)=\left[\RS{1^{h-j},2,1^{h}},\RS{1^{h-j}}\right),\,0\le j\le h,\\
\mathcal{I}_{2j} & =\left[\phi_{2j-1},\phi_{2j}\right)=\left[\RS{1^{h+1-j}},\RS{1^{h-j},2,1^{h}}\right),\,1\le j\le h,\\
\mathcal{R}_{k} & =\left[r_{k},R\right]=-\mathcal{R}_{-k},\,1\le k\le2h+1.\end{align*}
For $q=3$ we have $\np=1,$ $\phi_{0}=-\frac{1}{2}=\RS{2},$ $\phi_{1}=0$
and $r_{1}=r=\RS{\overline{3}}.$ Hence $\mathcal{I}_{1}=\left[-\frac{1}{2},0\right)=-\mathcal{I}_{-1}$
and $\mathcal{R}_{1}=\left[r,R\right]=-\mathcal{R}_{-1}$. 
\end{rem}
To establish the sought correspondence between the domain $\Omega^{\infty}$
and $\BIRED$ we first need a Lemma. 

\begin{lem}
\label{lem:r-is-the-smallest_prependablee_nr}$r$ is the smallest
number $y$ in $I_{R}$ such that $\bi\left(x,y\right)\in\BIRNCF$
for all $x\in\mathcal{I}_{\np}$.
\end{lem}
\begin{proof}
Let $q$ be even. We know from Lemma \ref{lem:R-is-given-by-lambda-fraction}
and its proof that $r=\DS{\overline{1^{h-1},2}}$, $-R=\DS{1,\overline{1^{h-1},2}}$
and $\phi_{\np-1}=\phi_{h-1}=\RS{1}$. Hence $\bi\left(\phi_{h-1},r\right)\in\BIRED$
and for $-R\le y<r$ then $\Cd\left(y\right)=\DS{1^{h},b_{h+1},\ldots}$
and $\bi\left(\phi_{h-1},y\right)$ contains the forbidden block $\FS{1^{h+1}}$. 

Let $q\ge5$ be odd. Then $r=\DS{\overline{1^{h},2,1^{h-1},2}}$,
$-R=\DS{1,\overline{1^{h-1},2,1^{h},2}}$ and $\phi_{\np-1}=\phi_{2h}=\RS{2,1^{h}}.$
Hence $\bi\left(\phi_{2h},r\right)\in\BIRNCF$ and if $-R\le y<r$
then $\Cd\left(y\right)=\DS{1^{h},2,1^{h},b_{2h+2},\ldots}$ and $\bi\left(\phi_{2h},y\right)$
contains the forbidden block $\FS{1^{h},2,1^{h},2}$. We have shown
that $r$ is the smallest number such that $\bi\left(\phi_{\np-1},r\right)$
does not contain a forbidden block and it is easy to show that also
$\bi\left(x,r\right)\in\BIRED$ for any $x\in\mathcal{I}_{\np}=\left[\phi_{\np-1},0\right)$.
The same argument applies to $q=3$.
\end{proof}
\begin{lem}
\label{lem:Omega_inf_iff_BIRNCF}$\left(x,y\right)\in\Omega^{\infty}$
$\Leftrightarrow$ $\bi\left(x,y\right)\in\BIRNCF$. 
\end{lem}
\begin{proof}
Just as in the proof of Lemma \ref{lem:r-is-the-smallest_prependablee_nr}
it is not hard to verify that $r_{j}$ is the smallest number in $I_{R}$
with a dual regular expansion which can be prepended to the regular
expansion of $\phi_{j}$ and hence of all $x\in\left[\phi_{j},0\right)$. 
\end{proof}
\begin{defn}
\label{def:def_omega_star}To determine the first return map we introduce
a ,,conjugate{}`` region $\Omega^{*}=\Sx\left(\Omega^{\infty}\right)$
where $\Sx\left(x,y\right)=\left(Sx,-y\right)$, i.e. setting $\mathcal{I}_{j}^{*}=S\left(\mathcal{I}_{j}\cap I_{j}^{\infty}\right),$
$\mathcal{R}_{j}^{*}=-\mathcal{R}_{j}\cap I_{R}^{\infty}$ and $\Omega_{j}^{*}=\mathcal{I}_{j}^{*}\times\mathcal{R}_{j}^{*}$
we get $\Omega^{*}=\cup\Omega_{j}^{*}.$ 
\end{defn}
Thus $\Omega$ with $\Omega^{\infty}$ as a dense subset is the domain
of the natural extension $\Fqx$ of $\Fq$. An example of $\Omega$
and $\Omega^{*}$ is given in Figure \ref{fig:Domains-Omega-and}.
See \cite{MR1402490} for another choice of a ,,conjugate{}`` $\Omega^{*}$
of $\Omega$ using the maps $\left(x,y^{-1}\right)$ and also \cite{MR1650073}
for the corresponding domain for the reduced Rosen fractions. %
\begin{figure}
\caption{\label{fig:Domains-Omega-and}Domains $\Omega$ and $\Omega^{*}$
for $q=7$}

\includegraphics[scale=0.4]{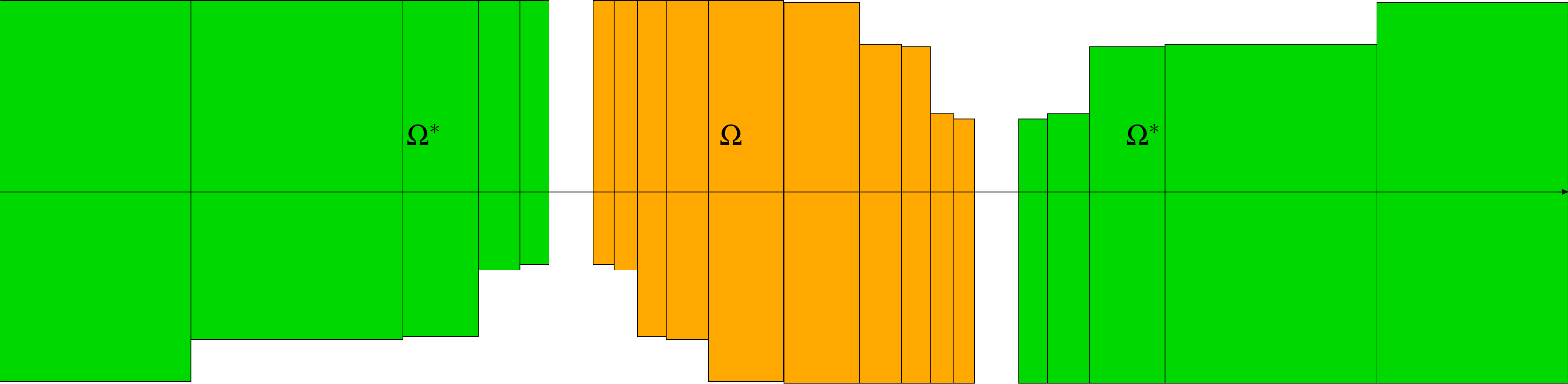}
\end{figure}

\subsection{Reduction of $\lambda$-fractions}

In a first step in our construction of a cross-section for the geodesic
flow we select a set of geodesics on $\H$ which contains at least
one lift of each geodesic on $\mathcal{M}_{q}=\H/G_{q}$, i.e. a set
of ,,representative{}`` or ,,reduced{}`` geodesics modulo $G_{q}$.
For an overview and a discussion of different reduction procedures
in the case of $\PSLZ$ see \cite[Sect.\ 3]{MR2265011}.

\begin{lem}
\label{lem:reduce_uv_to_Omega*}Let $u,v\in\mathbb{R}^{\infty}$ both
have infinite $\lambda$-fractions. Then there exists $B\in G_{q}$
such that $\left(Bu,Bv\right)\in\Omega^{*}$. 
\end{lem}
\begin{proof}
[Sketch of proof]The complete proof can be found in \cite{cf_main}.
First we find $A\in G_{q}$ such that $\Sx\left(Au,Av\right)\in I_{q}^{\infty}\times I_{R}^{\infty}$
by extending the domain of definition of $\Fqx^{-1}:\left(x,y\right)\mapsto\left(\frac{-1}{x+b_{1}\lambda},\Fqd y\right),$
$b_{1}=\nlms{-\frac{1}{y}}$ to $\mathbb{R}^{\infty}\times I_{R}^{\infty}$
and applying $\Fqx^{-1}$ repeatedly. Once the point is inside the
rectangle we attach to it a bi-infinite sequence $\bi\circ\Sx\left(Au,Av\right)=\FS{\ldots,b_{2},b_{1}\centerdot a_{1},a_{2},\ldots}.$
Using further backward shifts and rewriting of forbidden blocks we
obtain $\bi\circ\Sx\left(Bu,Bv\right)\in\BIRNCF$, i.e. $\left(Bu,Bv\right)\in\Omega^{*}.$
In the last step of rewriting we rely on the fact that rewriting does
not propagate to the left, cf.~Lemma \ref{lem:rewrite_not_affect_left}.
The explicit form of $B$ is given by $\Fqd\,^{n}y$ for some $n$. 
\end{proof}
In fact, one can do slightly better than in the previous Lemma by
using the important property of the number $r$, namely that $r$
and $-r$ are $G_{q}$-equivalent but not orbit-equivalent, i.e. $\mathcal{O}^{*}\left(r\right)\ne\mathcal{O}^{*}\left(-r\right)$
(cf.~Lemmas \ref{lem:properties_of_R_even} and \ref{lem:properties_of_R_odd}).
Using the explicit map identifying $r$ and $-r$ one can show that
it is possible to reduce any geodesic to one with endpoints in $\Omega^{*}$
without the upper horizontal boundary.

\begin{lem}
\label{lem:upper_and_lower_Omega*_equivalent}If $\left(x,y\right)\in\Omega^{*}$
and $y$ has a dual regular expansion with the same tail as $-r$
then there exists $A\in G_{q}$ such that $\left(Ax,Ay\right)\in\Omega^{*}$
and $Ay$ has the same tail as $r$. 
\end{lem}
\begin{proof}
Using $\Fqx$ we may assume that $y=-r$ and $Sx\in\mathcal{I}_{\np}$.
Consider an even $q$. There are two cases for $Sx\in\mathcal{I}_{h}:$
in $\Cq\left(Sx\right)=\RS{a_{1},a_{2},\ldots}$ either $a_{1}\ge2$
or $a_{1}=1,$ $a_{2}\le-1$. In the first case, set $A:=T^{-1}$,
$y':=Ay=-R,$ $x':=Ax$ and hence $Sx'=SAx\in\left(-\frac{\lambda}{2},0\right)$.
In the second case, let $A:=T^{-1}ST^{-1}$ with $A\left(-r\right)=r$
(recall that $-R=SR$). Set $y'=Ay=r$ and $x'=Ax$ with $\Cq\left(Sx'\right)=\RS{a_{2}-1,a_{3},\ldots}$
and $Sx'\in\mathcal{I}_{-h}$. 

For odd $q$ there are three cases to consider if $Sx\in\mathcal{I}_{2h+1}$
(cf.~Remark \ref{rem:For-odd-q-definitions}): We have $\Cq\left(Sx\right)=\RS{a_{1},a_{2},\ldots}$
with either $a_{1}\ge3$ or $\Cq\left(Sx\right)=\RS{2,1^{j},a_{j+2},\ldots}$
for some $0\le j\le h-1$ for $a_{j+2}\ne1$ or $\Cq\left(Sx\right)=\RS{2,1^{h},a_{h+2},\ldots}$
for $a_{h+2}\le-1$. In the first two cases one can use $A=T^{-1}$
and $-y'=R$ such that $Sx'\in\left[-\frac{\lambda}{2},0\right]$
and in the third case one can use $A=T^{-1}\left(ST^{-1}\right)^{h}ST^{-2}$
such that $A\left(-r\right)=r$ and hence $y'=r$ respectively $x'=Sx$
with $\Cq\left(Sx'\right)=\RS{a_{h+2}-1,a_{h+3},\ldots}$ and $SAx\in\mathcal{I}_{-\np}$. 

Hence in all cases $\left(Sx',-y'\right)\in\Omega^{\infty},$ i.e.
$\left(x',y'\right)\in\Omega^{*}$ and $y'=-R$ or $r$.  
\end{proof}

\subsection{Geodesics and geodesic arcs}

If $\gamma\left(\xi,\eta\right)$ is a geodesic in $\H$ oriented
from $\eta$ to $\xi$ (cf.~Section \ref{sub:Hyperbolic-geometry-and})
and $A\in\PSLR$ we define the geodesic $A\gamma$ as $A\gamma=\gamma\left(A\xi,A\eta\right).$
If $\left(\xi,\eta\right)\in\Omega^{*}$ we associate to $\left(\xi,\eta\right)$
a bi-infinite sequence (code) to $\gamma$, $\Bi\left(\gamma\right)=\Bi\circ\Sx\left(\xi,\eta\right)=\Cq\left(-\eta\right).\Cd\left(S\xi\right)\in\BIRNCF$. 

\begin{defn}
\label{def:reduced_geodesic}Let $\gamma$ be a geodesic on the upper
half-plane. We say that $\gamma$ is \emph{reduced }with respect to
$G_{q}$ if $\left(\gamma_{+},\gamma_{-}\right)\in\Omega^{*}$ and
$\gamma_{-}$ does not have the same tail as $-r$ or equivalently
if \[
\Bi\left(\gamma\right)\in\BIRED=\left\{ \underline{\zeta}\in\BIRNCF\:\large{|}\,\Fqx^{-n}\underline{\zeta}\ne\bi\left(\xi,-r\right)\:\forall\xi\in I_{q},\,\forall n\ge0\right\} .\]
We denote the set of reduced geodesics by $\Rgeo$ (or by $\Rgeo_{q}$
when we want to stress the dependence on $q$). 
\end{defn}
\begin{lem}
The coding map $\Bi:\Rgeo\rightarrow\BIRED$ is a homeomorphism. 
\end{lem}
\begin{proof}
$\Rgeo$ is identified with the set of endpoints in $\Omega^{*}$,
the map $\Sx:\Omega^{*}\rightarrow\Omega^{\infty}$ defined in Definition
\ref{def:def_omega_star} is a continuous bijection and by Corollary
\ref{cor:infinite-regular-are-unique} and \ref{cor:dual_regular_unique}
respectively Lemmas \ref{lem:Omega_inf_iff_BIRNCF} and \ref{lem:upper_and_lower_Omega*_equivalent}
it is clear that $\bi:\Omega^{\infty}\rightarrow\BIRED$ is continuous,
onto and one-to-one except for points where $\eta$ is equivalent
to $\pm r$ where it is two-to-one. Thus $\Bi=\bi\circ\Sx:\Rgeo\rightarrow\BIRED$
is a homeomorphism. For the case of $G_{q}=\PSLZ$ see also \cite[p.\ 19]{MR2265011}.
\end{proof}
The set of reduced geodesics contains representatives of all geodesics
on $\H/G_{q}$. This property is an immediate corollary to the following
Lemma which additionally also provides a reduction algorithm. 

\begin{lem}
Let $\gamma$ be a geodesic on the hyperbolic upper half-plane with
endpoints in $\mathbb{R}^{\infty}$ and $\gamma_{-}=\DS{b_{0};b_{1},\ldots}.$
Then there exists an integer $n\ge0$ and $A\in G_{q}$ such that
\[
A\, T^{-b_{n}}S\,\cdots\, T^{-b_{1}}ST^{-b_{0}}\gamma\in\Rgeo.\]
Here $A$ is one of the maps $Id$, $T^{-1}$ and $T^{-1}ST^{-1}$
for even $q$, respectively $T^{-1}\left(ST^{-1}\right)^{h}ST^{-2}$
for odd $q$.
\end{lem}
\begin{proof}
This is an immediate consequence of Lemmas \ref{lem:reduce_uv_to_Omega*}
and \ref{lem:upper_and_lower_Omega*_equivalent}. Note that $T^{-b_{0}}\gamma_{-}=\DS{b_{1},b_{2},\ldots}$
and $\Fqd\,^{n}T^{-b_{0}}\gamma_{-}=T^{-b_{n}}S\,\cdots\, T^{-b_{1}}ST^{-b_{0}}\gamma_{-}$. 
\end{proof}
\begin{cor}
\label{cor:any_geodesic_eq_to_reduced}If $\gamma^{*}$ is a geodesic
on $\H/G_{q}$ with all lifts having endpoints in $\mathbb{R}^{\infty}$
then $\Rgeo$ contains an element of $\pi^{-1}\left(\gamma^{*}\right)$.
\end{cor}
\begin{lem}
\label{lem:hyperbolic_are_periodic_fractions}If for $\xi\in I_{q}$
$\Cq\left(\xi\right)=\RS{\overline{a_{1},\ldots,a_{n}}}\in\ZRNCF$,
then $\xi$ is the attractive fixed point of the hyperbolic map $A=ST^{a_{1}}\, ST^{a_{2}}\,\cdots\, ST^{a_{n}}$
with conjugate fixed point $\xi^{*}=\eta^{-1}$, where $\Cd\left(\eta\right)=\DS{\overline{a_{n},a_{n-1},\ldots,a_{1}}}\in\ZDRNCF$.
Conversely, if $\xi$ is an hyperbolic fixed point of $B\in G_{q}$,
then $\Cq\left(\xi\right)$ is eventually periodic. 
\end{lem}
\begin{proof}
It is not hard to show, that $A\xi=\xi$ and $A\eta^{-1}=\eta^{-1}$.
Since $\eta\in I_{R},$ with $R\le1<\frac{2}{\lambda}$ it is clear
that $\xi^{*}\ne\xi$ and hence $A$ is hyperbolic. The other statement
in the Lemma is easy to verify by writing $B$ in terms of generators,
rewriting any forbidden blocks and going through all cases of non-allowed
sequences, e.g. if $B$ ends with an $S$. 
\end{proof}
Since the geodesic $\gamma\left(\xi,\eta\right)$ is closed if and
only if $\xi$ and $\eta$ are conjugate hyperbolic fixed points and
since $r$ and $-r$ are $G_{q}$-equivalent we conclude from Lemma
\ref{lem:hyperbolic_are_periodic_fractions} that there is a one-to-one
correspondence between closed geodesics on $\mathcal{M}_{q}=\H/G_{q}$
and the set of purely periodic regular $\lambda$-fractions which
do not have the same tail as $-r$. 

\begin{rem}
\label{rem:only_infinite_exp}Because $\Omega^{\infty}$ only contains
points with infinite $\lambda$-fractions, the set $\Rgeo$ does not
contain lifts of geodesics which \emph{disappear out to infinity},
i.e. with one or both endpoints equivalent to $\infty$. The neglected
set however corresponds to a set in $\SM$ of measure zero with respect
to any probability measure invariant under the geodesic flow. See
e.g. the second paragraph of \cite[p.~1]{MR2223106}.
\end{rem}
The subshift of finite type $\left(\BIRNCF,\sigma\right)$ is also
conjugate to the invertible dynamical system $\left(\Rgeo,\Fqx\right)$.
Here $\Fqx:\Rgeo\rightarrow\Rgeo$ is the map naturally induced by
$\Fqx$ acting on $\Omega^{\infty}$: i.e. if $\gamma=\gamma\left(\xi,\eta\right),$
then $\Fqx^{\pm1}\left(\gamma\right)=\gamma\left(\xi',\eta'\right)$
where $\left(\xi',\eta'\right)=\Sx\circ\Fqx^{\pm1}\circ\Sx^{-1}\left(\xi,\eta\right)$.
Using the same notation for both maps should not lead to any confusion.

\begin{defn}
For an oriented geodesic arc $c$ on $\H$ we let $\overline{c}$
denote the unique geodesic containing $c$ and preserving the orientation,
for instance $\overline{L_{1}}=\left\{ \smash{z=\frac{\lambda}{2}+iy\,\big{|}\, y>0}\right\} $
oriented upwards. Let $c^{\pm}$ denote the forward and backward end
points of $c$ and let $-c$ denote the geodesic arc with endpoints
$-c^{\pm}$. Here $-c$ should not to be confused with the geodesic
$c$ with reversed orientation, denoted by $c^{-1}$. 

For $z,w\in\H\cup\partial\H$ denote by $\left[z,w\right]$ the geodesic
arc oriented from $z$ to $w$ including the endpoints in $\H$. 
\end{defn}

\begin{defn}
Consider the fundamental domain $\mathcal{F}_{q}$ and its boundary
arcs $L_{0}$ and $L_{\pm1}$ as in Definition \ref{def:G_q_and_fundamental_domain}.
For the construction of the cross section we use the arcs $L_{\pm1}'=\left[\rho_{\pm},\pm\frac{\lambda}{2}\right]$,
i.e. $\overline{L}_{\pm1}=L_{\pm1}\cup L_{\pm1}^{'}$ respectively
$L_{\pm2}=\pm\left[\rho,\lambda\right]=T^{\pm1}SL_{\pm1}$ and $L_{\pm3}=\pm\left[\rho,\rho+\lambda\right]=T^{\pm1}L_{0}$
. 
\end{defn}

\section{\label{sec:Construction-of-the}Construction of the cross-section}

As a cross section for the geodesic flow on the unit tangent bundle
$\SM$ of $\mathcal{M}$ which can be identified with $\mathcal{F}_{q}\times S^{1}$
modulo the obvious identification of points on $\partial\mathcal{F}_{q}\times S^{1}$,
we will take a set of vectors with base points on the boundary $\partial\mathcal{F}_{q}$
directed inwards with respect to $\mathcal{F}_{q}$. The precise definition
will be given below. For a different approach to a cross section related
to a subgroup of $G_{q}$ see \cite{MR1424398}. For the sake of completeness
we include the case $q=3$ in our exposition but it is easy to verify
that our results in terms of the cross-section, first return map and
return time agree with the statements in \cite{MR2223106,MR2265011}.

\subsection{Strongly Reduced Geodesics}

\begin{defn}
\label{def:cross_section_sigma}We define the following subsets of
$\SM$: 

\begin{align*}
\Gamma_{r} & =\left\{ \left(z,\theta\right)\in L_{r}\times S^{1}\,|\,\dot{\gamma}_{z,\theta}\left(s\right)\,\mbox{directed inwards at }z\right\} ,\, r=0,\pm1,\\
\Sigma^{j} & =\left\{ \left(z,\theta\right)\in\Gamma_{j}\,|\,\gamma_{z,\theta}=\gamma\left(\xi,\eta\right)\in\Rgeo,\,\left|\xi\right|>\frac{3\lambda}{2}\,\mbox{or}\,\xi\eta<0\right\} ,\,-1\le j\le1,\\
\Sigma^{\pm2} & =\left\{ \left(z,\theta\right)\in\Gamma_{\pm1}\,|\,\gamma_{z,\theta}\notin\Rgeo,\,\gamma=T^{\pm1}S\gamma_{z,\theta}\in\Rgeo,\,\frac{3\lambda}{2}<\gamma_{+}\le\lambda+1\right\} ,\\
\Sigma^{\pm3} & =\left\{ \left(z,\theta\right)\in\Gamma_{0}\,|\,\gamma_{z,\theta}\notin\Rgeo,\,\gamma=T^{\pm1}\gamma_{z,\theta}\in\Rgeo\right\} .\end{align*}

\end{defn}
If $q$ is even let $\Sigma:=\cup_{j=-2}^{2}\Sigma^{j}$ and if $q\ge5$
is odd let $\Sigma:=\cup_{j=-3}^{3}\Sigma^{j}$. If $q=3$ we drop
the restriction on $\gamma_{+}$ in the definition of $\Sigma^{\pm2}$
and set $\Sigma=\cup_{j=-2}^{2}\Sigma^{j}$. 

We will show that there is a one-to-one correspondence between $\Sigma$
and a subset of reduced geodesics, which we call strongly reduced.
For $q=3$ these sets are identical.

\begin{defn}
\label{def:strongly_reduced}A reduced geodesic $\gamma\left(\xi,\eta\right)\in\Rgeo$
is said to be \emph{strongly reduced} if $\left|\xi\right|>\frac{3\lambda}{2}$
or $\xi\eta<0$. Denote by $\Sgeo$ the set of strongly reduced geodesics
and by $\Sred$ the corresponding set of $\left(\xi,\eta\right)\in\Omega^{*}$.
Then $\Sredx:=\Sx\left(\Sred\right)\subseteq\Omega^{\infty}$ and
$\Sredo=\overline{\Sredx}$, the closure of $\Sredx$ in $\mathbb{R}^{2}$,
i.e. $\Sredx=\Sredo\cap\Omega^{\infty}$, and $\SBIRED=\bi\left(\Sgeo\right)\subseteq\BIRED.$
\end{defn}

\begin{rem}
We observe that for odd $q$ by Lemma \ref{lem:expansion_of_1_odd_q}
$\phi_{\np-1}=\RS{2,1^{h}}=\frac{-1}{\lambda+1}$ and thus $-\frac{2}{3\lambda}<\phi_{\np-1}$.
For even $q$ on the other hand $\phi_{\np-1}=\frac{-1}{\lambda}$
and since $\frac{-1}{\lambda}<\frac{-2}{3\lambda}$ we have $\phi_{\np-1}<-\frac{2}{3\lambda}$.
Hence the shape of $\Sredx=\left\{ \left(u,v\right)\in\Omega^{\infty}\,|\,\left|u\right|\le\frac{2}{3\lambda}\,\mbox{or\,}uv<0\right\} $
differs slightly between even and odd $q$. Set $\Lambda_{1}:=\left(-\frac{\lambda}{2},0\right)\times\left[0,R\right]$
and $\Lambda_{2}:=\left(0,\frac{2}{3\lambda}\right)\times\left[0,-r\right]$
for even $q$ respectively $\Lambda_{2}:=\left(0,\phi_{\np-1}\right)\times\left[0,-r\right]$
and $\Lambda_{3}:=\left(\phi_{\np-1},\frac{2}{3\lambda}\right)\times\left[0,r_{\np-1}\right]$
for odd $q$. Then we have: %
\begin{figure}
\caption{\label{fig:Sredx}Domain of Strongly reduced geodesics $\Sredo$}

\includegraphics{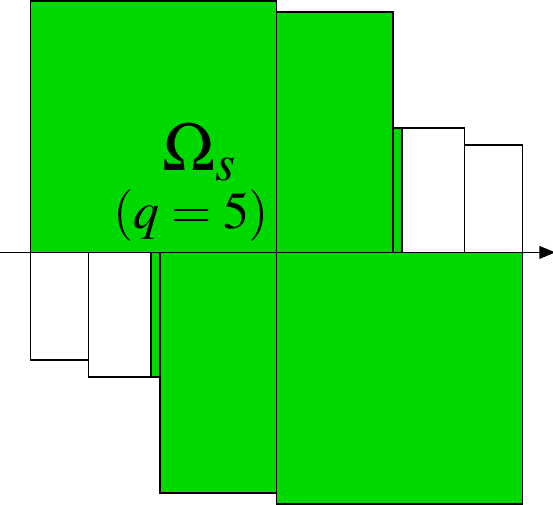}\includegraphics{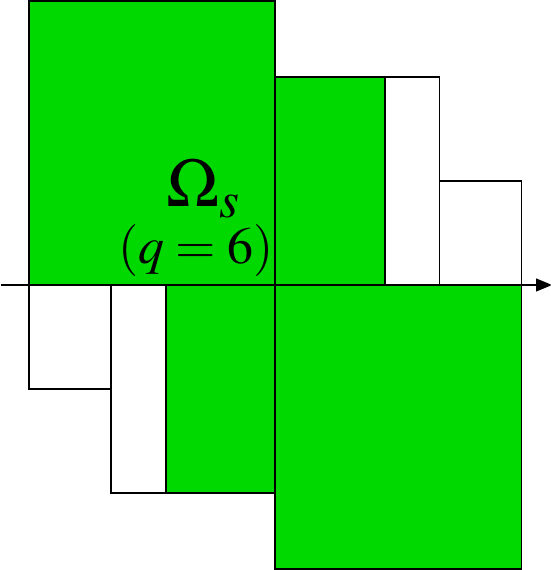}
\end{figure}
 \[
\Sredo=\bigcup_{j=1}^{k}\Lambda_{j}\cup-\Lambda_{j},\]
with $k=2$ for even and $3$ for odd $q$. See also Figure \ref{fig:Sredx}
where $\Sredo$ is displayed for $q=5$ and $q=6$ as a subset of
$\Omega$. An even more convenient description of the set of strongly
reduced geodesics is in terms of the bi-infinite codes of their base
points $\left(\xi,\eta\right)$ \[
\SBIRED=\left\{ \RS{\ldots b_{2},b_{1}\centerdot a_{0},a_{1},\ldots}\in\BIRED\,\big{|}\,\left|a_{0}\right|\ge2,\,\mbox{or}\,\,\, a_{0}b_{1}>0\right\} .\]

\end{rem}
\begin{lem}
\label{lem:If-a0ge2then_strongly_reduced}There exists a bijection
$\PP:\Sgeo\rightarrow\Sigma$ defined through $\PP\left(\gamma\right):=\left(z,\theta\right)\in\Sigma$
with $z=\gamma\left(s\right)\in\partial\mathcal{F}_{q}$ for some
$s\in\mathbb{R}$ and $\theta$ given by $\Arg\dot{\gamma}\left(s\right)=\theta.$ 
\end{lem}
\begin{proof}
Consider Figure \ref{fig:Cross-Section} %
\begin{figure}
\caption{\label{fig:Cross-Section}Cross-Section ($q=5$)}

\includegraphics[scale=0.45]{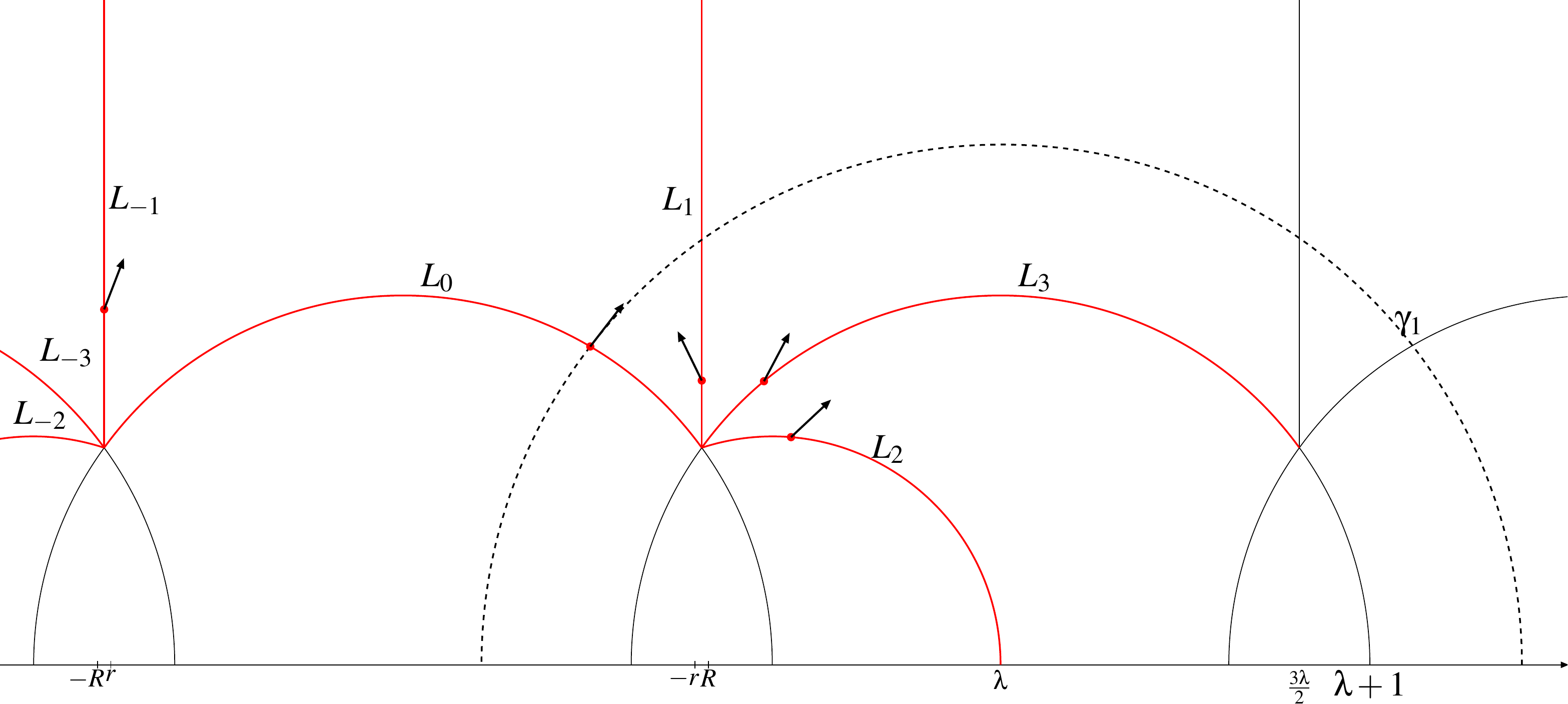}
\end{figure}
 and a geodesic $\gamma_{1}\left(\xi,\eta\right)$ from $\eta\in\left[-R,-r\right)$
to $\xi>\frac{3\lambda}{2}$. It is clear that either $\gamma_{1}$
intersects $L_{-1}\cup L_{0}$ inwards or $L_{2}$ from the left to
the right. Let $z=\gamma\left(s\right)$ be this intersection and
set $\theta=\Arg\dot{\gamma}\left(s\right)$. In the first case we
get $\PP\gamma=\left(z,\theta\right)\in\Sigma^{-1}\cup\Sigma^{0}$.
In the second case we either get $\PP\gamma=\left(z,\theta\right)\in\Sigma^{2}$
if $\xi\in\left(\frac{3\lambda}{2},\lambda+1\right)$ or, if $q$
is odd and $\xi>\lambda+1,$ we get $\PP\gamma=\left(z',\theta'\right)\in\Sigma^{3}$,
where $z'=\gamma\left(s'\right)\in L_{3}$ and $\theta'=\Arg\dot{\gamma}\left(s'\right)$,
since by Lemma \ref{lem:odd_reduced_does_not_intersect_l} $\gamma$
must intersect $L_{3}$. Remember, that for $q$ even, $r=1-\lambda$,
so $\gamma$ can not intersect $\overline{L}_{3}$ to the right of
$\frac{\lambda}{2}$. Consider next a geodesic $\gamma\left(\xi,\eta\right)$
with $\eta\in\left[-R,0\right)$ and $\xi\in\left(\frac{2}{\lambda},\frac{3\lambda}{2}\right)$.
Since the geodesic $S\overline{L}_{-1}$ from $0$ to $\frac{2}{\lambda}$
intersects $\rho$, the intersection point of $\gamma$ and $\overline{L}_{1}$
must lie above $\rho$. Hence $\gamma$ intersects either $L_{-1}$
or $L_{0}$ inwards, i.e. $\PP\gamma\in\Sigma^{-1}\cup\Sigma^{0}$. 

The case $\xi<0$ is analogous and the inverse map $\PP^{-1}:\Sigma\rightarrow\Sgeo$
is clearly given by $\PP^{-1}\left(z,\theta\right)=\gamma_{z,\theta}$
if $\left(z,\theta\right)\in\Sigma^{j},$ $\left|j\right|\le1,$ respectively
$\PP^{-1}\left(z,\theta\right)=T^{\pm1}S\gamma_{z,\theta}$ if $\left(z,\theta\right)\in\Sigma^{\pm2}$
and $\PP^{-1}\left(z,\theta\right)=T^{\pm1}\gamma_{z,\theta}$ if
$\left(z,\theta\right)\in\Sigma^{\pm3}$. 
\end{proof}
\begin{defn}
For $\left(\xi,\eta\right)\in\Sred$ we define $\PT:\Sred\rightarrow\Sigma$
by,  $\PT\left(\xi,\eta\right):=\PP\gamma\left(\xi,\eta\right)$.
\end{defn}
\begin{lem}
\label{lem:gamma_reduced_FkTa_gamma-strongly_reduced}If $\gamma$
is a reduced geodesic on $\H$ then there exists an integer $k\ge0$
such that $\Fqx^{k}\gamma$ is strongly reduced.
\end{lem}
\begin{proof}
For $k\ge0$ let $\gamma^{k}:=\Fqx^{k}\gamma$, $\Cq\left(\gamma_{+}^{k}\right)=\RS{a_{k};a_{k+1},\ldots}$
and $\Cd\left(\gamma_{-}^{k}\right)=\DS{a_{k-1},\ldots,a_{0},b_{1},\ldots}$.
If $\left|a_{k}\right|\ge2$ then $\left|\gamma_{+}^{k}\right|>\frac{3\lambda}{2}$
and if $a_{k-1}a_{k}<0$ then $\gamma_{+}^{k}\gamma_{-}^{k}<0$. In
both cases by definition $\Fqx^{k}\gamma\in\Sgeo$. Since an infinite
sequence of $1$'s or $-1$'s is forbidden, it is clear that there
exists a $k\ge0$ such that one of these conditions apply.
\end{proof}
Combining the above Lemma with Lemma \ref{lem:reduce_uv_to_Omega*}
we have shown, that every reduced geodesic is $G_{q}$-equivalent
to a strongly reduced geodesic. 

A consequence of Lemma \ref{lem:gamma_reduced_FkTa_gamma-strongly_reduced}
is that for any strongly reduced geodesic we can find an infinite
number of strongly reduced geodesics in its forward and backward $\Fqx$-orbit
(with infinite repetitions if the geodesic is closed). Furthermore,
since the base-arcs $L_{\pm1}$ and $L_{0}$ of $\Sigma$ consist
of geodesics, none of whose extensions are in $\Sgeo,$ it is clear
that any strongly reduced geodesic intersects $\Sigma$ transversally.
The set $\Sigma$ thus fulfills the requirements (P1) and (P2) of
Definition \ref{def:poincare_section} and is a Poincaré (or cross-)
section with respect to $\Sgeo$. Since any geodesic $\gamma^{*}$
on $\mathcal{M}_{q}$ which does not go into infinity has a strongly
reduced lift we also have the following lemma. 

\begin{lem}
\label{lem:Sigma_is_Poincare_section}$\pi^{*}\left(\Sigma\right)$
is a Poincaré section for the part of the geodesic flow on $\SM$
which does not disappear into infinity. 
\end{lem}
From the identification of $\Sigma$ and $\Sred$ via the map $\PT$
we see that the natural extension $\Fqx$ of the continued fraction
map $\Fq$ induces a \emph{return map} for $\Sigma$, i.e. if $\mathbf{z}=\left(z,\theta\right)\in\Sigma$
then $\PT\circ\Fqx^{n}\circ\PT^{-1}\mathbf{z}\in\Sigma$ for an infinite
number of $n\ne0$. We give a geometric description of the first return
map for $\Sigma$ and we will later see that this map is in fact also
induced by $\Fqx$.

\begin{defn}
\label{def:First_return_map}The \emph{first return map} $\RET:\Sigma\rightarrow\Sigma$
is defined as follows (cf.~Figure \ref{fig:Return-map}): %
\begin{figure}
\caption{\label{fig:Return-map}Illustration of the first return map}

\includegraphics[scale=0.45]{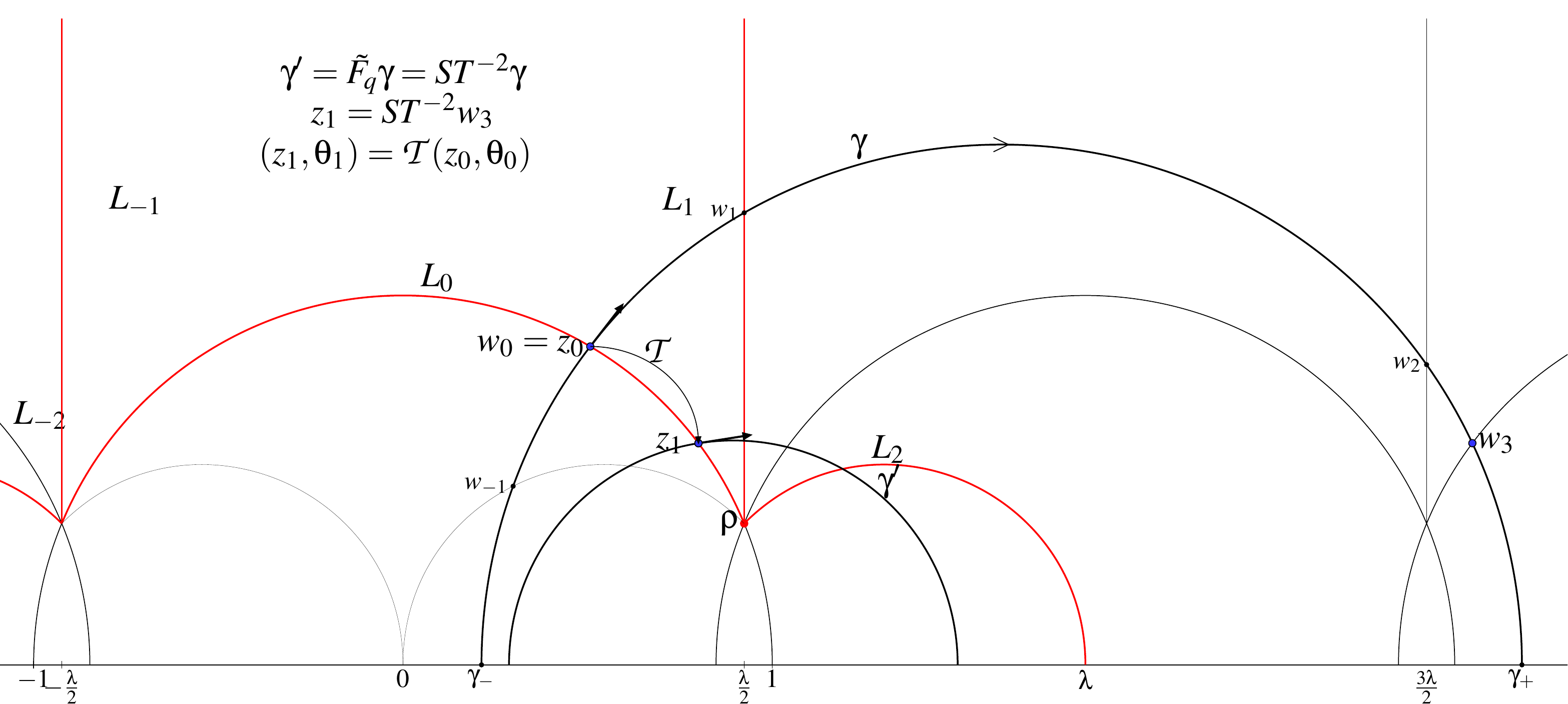}
\end{figure}
If $\mathbf{z}_{0}\in\Sigma$ and $\gamma=\PP^{-1}\mathbf{z}_{0}\in\Sgeo$
let $\left\{ w_{n}\right\} _{n\in\mathbb{Z}}$ be the ordered sequence
of intersections in the direction from $\gamma_{-}$ to $\gamma_{+}$
between $\gamma$ and the $G_{q}$-translates of $\partial\mathcal{F}$
with $w_{0}$ given by $z_{0}$. Since $\gamma_{+}$ and $\gamma_{-}$
have infinite $\lambda$-fractions they are not cusps of $G_{q}$
and the sequence $w_{n}$ is bi-infinite. For each $w_{n}$ let $A_{n}\in G_{q}$
be the unique map such that $w_{n}'=A_{n}w_{n}\in\partial\mathcal{F}$
and $\gamma'=A_{n}\gamma$ intersects $\partial\mathcal{F}$ at $w_{n}'$
in the inwards direction. 

If $\gamma'\in\Sgeo$ and $\PP\gamma'=\mathbf{z}'$ we say that $\mathbf{z}'\in\Sigma$
is a \emph{return} of $\gamma$ to $\Sigma$. If $n_{0}>0$ is the
smallest integer such that $w_{n_{0}}$ gives a return to $\Sigma$
we say that the corresponding point $\PP A_{n_{0}}\gamma=\mathbf{z}_{1}\in\Sigma$
is the \emph{first return} and the \emph{first return map} $\RET:\Sigma\rightarrow\Sigma$
is defined by $\RET\mathbf{z}_{0}=\mathbf{z}_{1}$ where $\mathbf{z}_{1}$
is the first return after $\mathbf{z}_{0}$. Sometimes $\RETx:\Sred\rightarrow\Sred$
given by $\RETx=\PT^{-1}\circ\RET\circ\PT$ is also called the \emph{first
return map.} 
\end{defn}
After proving some useful geometric lemmas in the next section we
will show in Section \ref{sub:First-return-map} that the first return
map $\RET$ is given explicitly by powers of $\Fqx.$

\subsection{\label{sub:Geometric-Lemmas}Geometric Lemmas}

\begin{lem}
\label{lem:z_theta->gamma_z_theta_unique_cont}The map $\mathbf{z}=\left(z,\theta\right)\mapsto\left(\gamma_{z,\theta},s\right)\cong\left(\xi,\eta,s\right)$
where $\gamma_{z,\theta}=\gamma\left(\xi,\eta\right),$ $\gamma_{z,\theta}\left(s\right)=z$
and $\Arg\dot{\gamma}_{z,\theta}\left(s\right)=\theta$ is a diffeomorphism
for $\theta\ne\pm\frac{\pi}{2}$. 
\end{lem}
\begin{proof}
Let $z=x+iy,$ $y>0,$ and $\theta\in\left[-\pi,\pi\right)$ be given.
First we want to show that there exist $\xi,\eta\in\mathbb{R}^{*}$
and $s\in\mathbb{R}$ such that for the geodesic $\gamma=\gamma\left(\xi,\eta\right)$
one finds $\gamma\left(s\right)=z$ and $\Arg\dot{\gamma}\left(s\right)=\theta$.
Without loss of generality we may assume that $\theta\in\left(-\frac{\pi}{2},\frac{\pi}{2}\right)$
so that $\eta<\xi$. Set $c=\frac{1}{2}\left(\eta+\xi\right),$ $r=\frac{1}{2}\left(\xi-\eta\right)$
and parametrize $\gamma$ as $\gamma\left(t\right)=c+re^{it},$ $0<t<\pi$.
It is easy to verify that if $c=x+y\tan\theta$, $r=\frac{y}{\cos\theta}$
and $t_{0}=\theta+\frac{\pi}{2}$ then $\gamma\left(t_{0}\right)=z$
and $\Arg\dot{\gamma}\left(t_{0}\right)=\theta$. See Figure \ref{fig:gamma_z,theta}.
To find the arc length parameter $s$ we use the isometry $A:z\mapsto\frac{-z+\left(c-r\right)}{z-\left(c+r\right)}$
mapping $\gamma$ to $i\mathbb{R}^{+},$ $A\gamma\left(t\right)=i\left(\tan\frac{t}{2}\right)^{-1}$
and $A\left(c+ir\right)=i$. It is then an easy computation to see
that $s\left(\theta\right)=d\left(z,c+ir\right)=d\left(i/\tan\frac{t}{2},i\right)=\ln\tan\frac{t}{2}=\ln\tan\left(\frac{\theta}{2}+\frac{\pi}{4}\right)$.
From the above formulas one can easily deduce differentiability of
the map $\left(z,\theta\right)\rightarrow\left(\xi,\eta,s\right)$
as well as its inverse away from $\theta=\pm\frac{\pi}{2}$.

\begin{figure}
\caption{\label{fig:gamma_z,theta}}
\includegraphics[scale=0.75]{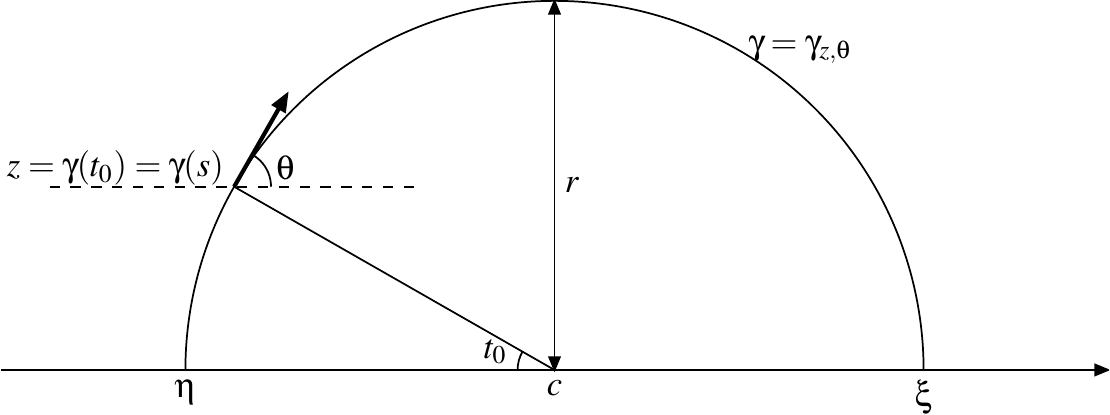}
\end{figure}

\end{proof}
\begin{cor}
\label{cor:change_var_x-y-theta->xi-eta-s}The map $\left(x,y,\theta\right)\rightarrow\left(\xi,\eta,s\right)$
of Lemma \ref{lem:z_theta->gamma_z_theta_unique_cont} gives a change
of variables on $\SH$ which is diffeomorphic away from $\theta=\pm\frac{\pi}{2}$.
Explicitly, $\xi=x+y\tan\theta+\frac{y}{\left|\cos\theta\right|}$,
$\eta=x+y\tan\theta-\frac{y}{\left|\cos\theta\right|},$ $s=\ln\tan\left(\frac{\theta}{2}+\frac{\pi}{4}\right)$
and the corresponding Jacobian is $\left|\frac{\partial\left(x,y,\theta\right)}{\partial\left(\xi,\eta,s\right)}\right|=\frac{1}{2}\cos^{2}\theta$. 
\end{cor}
\begin{proof}
This follows from the proof of Lemma \ref{lem:z_theta->gamma_z_theta_unique_cont}
and a trivial computation.
\end{proof}
\begin{defn}
\label{def:g(z,x)}For $z\in\H$ and $\xi\in\mathbb{R}$ define \[
g\left(z,\xi\right)=\frac{\left|z-\xi\right|^{2}}{\Im z}\]
and for $\gamma$ a geodesic with endpoints $\gamma_{+}$ and $\gamma_{-}$
set $g\left(z,\gamma\right)=g\left(z,\gamma_{+}\right).$
\end{defn}
\begin{lem}
\label{lem:g(Az,Aw)=00003D}If $A\in\PSLR$ then $g\left(Az,A\xi\right)=g\left(z,\xi\right)A'\left(\xi\right).$ 
\end{lem}
\begin{proof}
Let $A=\left(\begin{smallmatrix}a & b\\
c & d\end{smallmatrix}\right),$ then $\Im Az=\frac{\Im z}{\left|cz+d\right|^{2}}$, $\left|Az-A\xi\right|^{2}=\frac{\left|z-\xi\right|^{2}}{\left|cz+d\right|^{2}\left|c\xi+d\right|^{2}}$
and since $\xi\in\mathbb{R}$ \[
g\left(Az,Ax\right)=g\left(z,\xi\right)\left(c\xi+d\right)^{-2}=g\left(z,\xi\right)A'\left(\xi\right).\]

\end{proof}
\begin{lem}
\label{lem:hyp_dist_points_on_geod}Let $\gamma=\gamma\left(\xi,\eta\right)$
be a geodesic with $\xi,\eta\in\mathbb{R}$, $\eta<\xi$ and suppose
that $z_{j}$, $j=1,2,$ with $\eta\le\Re z_{1}<\Re z_{2}\le\xi$
are two points on $\gamma$. Then \[
d\left(z_{1},z_{2}\right)=\ln g\left(z_{1},\xi\right)-\ln g\left(z_{2},\xi\right)=\ln\left[\frac{\Im z_{2}}{\Im z_{1}}\left|\frac{z_{1}-\xi}{z_{2}-\xi}\right|^{2}\right].\]

\end{lem}
\begin{proof}
It is easy to verify, that the hyperbolic isometry $B:w\mapsto\frac{w-\eta}{\xi-w}$
maps $\gamma$ to $i\mathbb{R}^{+}$ and if $a<b$ then $d\left(ia,ib\right)=\int_{a}^{b}\frac{dy}{y}=\ln\left(\frac{b}{a}\right)$.
Thus if $\Im z_{j}=y_{j}$ we get \[
d\left(z_{1},z_{2}\right)=d\left(Bz_{1},Bz_{2}\right)=d\left(i\Im Bz_{1},i\Im Bz_{2}\right)=\ln\left[\frac{y_{2}}{y_{1}}\left|\frac{z_{1}-\xi}{z_{2}-\xi}\right|^{2}\right].\]

\end{proof}
\begin{lem}
\label{lem:intersection_gamma_iR}Let $\gamma=\gamma\left(\xi,\eta\right)$
be a geodesic with $\xi,\eta\in\mathbb{R}$, $\eta<0<\xi$ and let
$z=\gamma\cap i\mathbb{R}$ be the intersection of $\gamma$ with
the imaginary axis. Then \[
z=i\sqrt{-\xi\eta}.\]

\end{lem}
\begin{proof}
With $r=\frac{1}{2}\left(\xi-\eta\right)$ and $c=\frac{1}{2}\left(\xi+\eta\right)$
any point on $\gamma$ is given by $\gamma\left(t\right)=c+re^{it}$
for some $0\le t\le\pi$. Suppose $z=\gamma\left(t_{0}\right)$, then
$\Re\gamma\left(t_{0}\right)=c+r\cos t_{0}=0$ and hence $\cos t_{0}=-\frac{c}{r}$.
But then $\sin^{2}t_{0}=1-\frac{c^{2}}{r^{2}}$ and therefore $z=ir\sin t_{0}=i\sqrt{r^{2}-c^{2}}=i\sqrt{-\xi\eta}$.
\end{proof}
\begin{lem}
Let $\gamma=\gamma\left(\xi,\eta\right)$ be a geodesic with $\xi,\eta\in\mathbb{R}$,
$\eta<\xi$. For $\omega$ a geodesic intersecting $\gamma$ at $w\in\H$
let $A=\left(\begin{smallmatrix}a & b\\
c & d\end{smallmatrix}\right)\in\PSLR$ be such that $A\omega=i\mathbb{R}^{+}$. Then \begin{align*}
w=w\left(\xi,\eta\right) & =\frac{1}{ad+bc+ac\left(\xi+\eta\right)}\left[ac\xi\eta-bd+\epsilon i\sqrt{-l_{A}\left(\xi\right)l_{A}\left(\eta\right)}\right],\,\mbox{and}\\
g\left(w,\gamma\right) & =\frac{\left|w-\xi\right|^{2}}{\Im w}=\left(\xi-\eta\right)\sqrt{-\frac{l_{A}\left(\xi\right)}{l_{A}\left(\eta\right)}}\end{align*}
where $l_{A}\left(\xi\right)=\left(a\xi+b\right)\left(c\xi+d\right)$
and $\epsilon=\sign\left(ad+bc+ac\left(\xi+\eta\right)\right)$. 
\end{lem}
\begin{proof}
According to Lemma \ref{lem:g(Az,Aw)=00003D} $g\left(w,\gamma\right)=g\left(A^{-1}z,A^{-1}\gamma'\right)=g\left(z,\gamma'\right)A^{-1}\,'\left(A\xi\right)$
where $\gamma'=A\gamma$ and $z=Aw\in i\mathbb{R}.$ By Lemma \ref{lem:intersection_gamma_iR}
$z=i\sqrt{-\xi'\eta'}$ with $\xi'=A\xi$ and $\eta'=A\eta$. We choose
$A$, i.e. the orientation of $\gamma'$ such that $\eta'<0<\xi'$
and hence $\sgn\left(a\xi+b\right)=\sgn\left(c\xi+d\right)$ and $\sgn\left(a\eta+b\right)=-\sgn\left(c\eta+d\right)$.
Then \begin{align*}
w & =A^{-1}\left(i\sqrt{-A\xi A\eta}\right)=\left(di\sqrt{\left|\frac{a\xi+b}{c\xi+d}\frac{a\eta+b}{c\eta+d}\right|}-b\right)\left(a-ci\sqrt{\left|\frac{a\xi+b}{c\xi+d}\frac{a\eta+b}{c\eta+d}\right|}\right)^{-1}\\
 & =\left[i\sqrt{\left|\frac{a\xi+b}{c\xi+d}\frac{a\eta+b}{c\eta+d}\right|}-\left(ab+dc\left|\frac{a\xi+b}{c\xi+d}\frac{a\eta+b}{c\eta+d}\right|\right)\right]\left(c^{2}\left|\frac{a\xi+b}{c\xi+d}\frac{a\eta+b}{c\eta+d}\right|+a^{2}\right)^{-1}\\
 & =\frac{-\left(ab\left|c\xi+d\right|\left|c\eta+d\right|+dc\left|a\xi+b\right|\left|a\eta+b\right|\right)+i\sqrt{-l_{A}\left(\xi\right)l_{A}\left(\eta\right)}}{c^{2}\left|\left(a\xi+b\right)\left(a\eta+b\right)\right|+a^{2}\left|c\xi+d\right|\left|c\eta+d\right|}\\
 & =\frac{\epsilon\left(ab\left(c\xi+d\right)\left(c\eta+d\right)-dc\left(a\xi+b\right)\left(a\eta+b\right)\right)+i\sqrt{-l_{A}\left(\xi\right)l_{A}\left(\eta\right)}}{c^{2}\epsilon\left(a\xi+b\right)\left(a\eta+b\right)-a^{2}\epsilon\left(c\xi+d\right)\left(c\eta+d\right)}\\
 & =\frac{i\sqrt{-l_{A}\left(\xi\right)l_{A}\left(\eta\right)}-\epsilon\left[\xi\eta ca-bd\right]}{-\epsilon\left[\left(\xi+\eta\right)ca+\left(cb+ad\right)\right]}=\frac{ac\xi\eta-bd-i\epsilon\sqrt{-l_{A}\left(\xi\right)l_{A}\left(\eta\right)}}{ac\left(\xi+\eta\right)+\left(ad+bc\right)}\end{align*}
where $\epsilon=\sign\left(\left(a\xi+b\right)\left(a\eta+b\right)\right)=-\sign\left(ac\left(\xi+\eta\right)+ad+bc\right)$
since $w\in\H.$ For the function $g$ we now have \[
g\left(z,\gamma'\right)=\frac{\left|z-\xi'\right|^{2}}{\Im z}=\frac{\xi'^{2}-\xi'\eta'}{\sqrt{-\xi'\eta'}}=\left(\xi'-\eta'\right)\sqrt{\frac{\xi'}{-\eta'}}.\]
Since $A^{-1}\,'\left(A\xi\right)=\left(-cA\xi+a\right)^{-2}=\left(c\xi+d\right)^{2}$
Lemma \ref{lem:g(Az,Aw)=00003D} implies that \begin{align*}
g\left(w,\gamma\right) & =\left(\xi'-\eta'\right)\sqrt{\frac{\xi'}{-\eta'}}\left(c\xi+d\right)^{2}=\left(\frac{a\xi+b}{c\xi+d}-\frac{a\eta+b}{c\eta+d}\right)\sqrt{\frac{\frac{a\xi+b}{c\xi+d}}{\frac{a\eta+b}{c\eta+d}}}\left(c\xi+d\right)^{2}\\
 & =\left(\left(a\xi+b\right)\left(c\eta+d\right)-\left(a\eta+b\right)\left(c\xi+d\right)\right)\sqrt{\frac{\left(a\xi+b\right)\left(c\xi+d\right)}{\left(a\eta+b\right)\left(c\eta+d\right)}}\\
 & =\left(\xi-\eta\right)\sqrt{\frac{\left(a\xi+b\right)\left(c\xi+d\right)}{\left(a\eta+b\right)\left(c\eta+d\right)}}=\left(\xi-\eta\right)\sqrt{-\frac{l_{A}\left(\xi\right)}{l_{A}\left(\eta\right)}}.\end{align*}
~
\end{proof}
Application of the previous Lemma to vertical or circular geodesics
yields the following corollaries:

\begin{cor}
\label{lem:intersection_vertical_and_circular}Let $\gamma\left(\xi,\eta\right)$
be a geodesic with $\eta<\xi$. Then

a) if $\eta<a<\xi$ for some $a\in\mathbb{R}$, then $\gamma$ and
the vertical geodesic $\omega_{v}=a+i\mathbb{R}^{+}$ intersect at
\begin{align*}
\mathcal{Z}_{v}\left(\xi,\eta\right) & =a+i\sqrt{\left(\xi-a\right)\left(-\eta+a\right)}\in\H,\,\mbox{and}\\
g_{v}\left(\xi,\eta\right) & =g\left(\mathcal{Z}_{v},\gamma\right)=\left(\xi-\eta\right)\sqrt{-\frac{\xi-a}{\eta-a}};\end{align*}
b) if $\eta<c-\rho<\xi<c+\rho$ for some $c\in\mathbb{R}$ and $\rho\in\mathbb{R}^{+}$,
then $\gamma$ and the circular geodesic $\omega_{c}$ with center
$c$ and radius $\rho$ intersect at \begin{align*}
\mathcal{Z}_{c}\left(\xi,\eta\right) & =\frac{\xi\eta+\rho^{2}-c^{2}}{\xi+\eta-2c}+\frac{i}{\left|\xi+\eta-2c\right|}\sqrt{\left(\left(\xi-c\right)^{2}-\rho^{2}\right)\left(\rho^{2}-\left(\eta-c\right)^{2}\right)}\in\H,\,\mbox{and}\\
g_{c}\left(\xi,\eta\right) & =g\left(\mathcal{Z}_{c},\gamma\right)=\left(\xi-\eta\right)\sqrt{-\frac{\left(\xi-c\right)^{2}-\rho^{2}}{\left(\eta-c\right)^{2}-\rho^{2}}}.\end{align*}
The subscripts ,,$v${}`` and ,,$c${}`` above refer to intersections
with \emph{vertical} and \emph{circular} geodesics respectively.
\end{cor}

\begin{cor}
\label{cor:_explicit_formulas_for_g023}Let $\gamma=\gamma\left(\xi,\eta\right)$
be an arbitrary geodesic on $\H$ with $\xi,\eta\in\mathbb{R}.$ For
$\mathcal{Z}_{j}\left(\xi,\eta\right)=\gamma\cap\overline{L_{j}}$
set $g_{j}\left(\xi,\eta\right)=g\left(\mathcal{Z}_{j},\gamma\right)$,
$-2\le j\le2$. If $\mathcal{Z}_{j}\left(\xi,\eta\right)$ exists,
the following formulas hold: \begin{align*}
\mathcal{Z}_{0}\left(\xi,\eta\right) & =\frac{1}{\xi+\eta}\left(1+\xi\eta+\epsilon i\sqrt{\left(\xi^{2}-1\right)\left(1-\eta^{2}\right)}\right),\,\epsilon=\sgn\left(\xi+\eta\right),\\
\mathcal{Z}_{\pm1}\left(\xi,\eta\right) & =\mp\frac{\lambda}{2}+i\sqrt{\left(\xi\mp\frac{\lambda}{2}\right)\left(-\eta\pm\frac{\lambda}{2}\right)},\\
\mathcal{Z}_{\pm2,\pm3}\left(\xi,\eta\right) & =\frac{1}{\xi+\eta\mp2c}\left(\xi\eta+\rho^{2}-c^{2}+\epsilon i\sqrt{\left(\left(\xi\mp c\right)^{2}-\rho^{2}\right)\left(\rho^{2}-\left(\eta\mp c\right)^{2}\right)}\right)\\
 & \,\mbox{where }\epsilon=\sgn\left(x+y\mp2c\right).\end{align*}
Furthermore\begin{align*}
g_{0}\left(\xi,\eta\right) & =\left(\xi-\eta\right)\sqrt{\frac{\xi^{2}-1}{1-\eta^{2}}},\,\\
g_{\pm1}\left(\xi,\eta\right) & =\left(\xi-\eta\right)\sqrt{-\frac{\xi\mp\frac{\lambda}{2}}{\eta\mp\frac{\lambda}{2}}},\\
g_{\pm j}\left(\xi,\eta\right) & =\left(\xi-\eta\right)\sqrt{-\frac{\left(\xi\mp c\right)^{2}-\rho^{2}}{\left(\eta\mp c\right)^{2}-\rho^{2}}},\, j=2,3.\end{align*}
Here \textup{$\left(\rho,c\right)=\left(\lambda-\frac{1}{\lambda},\frac{1}{\lambda}\right)$
for $j=2$ and $\left(1,\lambda\right)$ for $j=3$. }
\end{cor}
\begin{proof}
Taking $A_{0}=\frac{1}{\sqrt{2}}\left(\begin{smallmatrix}1 & -1\\
1 & 1\end{smallmatrix}\right),$ $A_{\pm1}=\left(\begin{smallmatrix}1 & \mp\frac{\lambda}{2}\\
0 & 1\end{smallmatrix}\right)$ and $A_{\pm2}=\frac{1}{\sqrt{2\rho}}\left(\begin{smallmatrix}-1 & \pm c-\rho\\
1 & \mp c-\rho\end{smallmatrix}\right)$ it is easy to verify that $A_{j}\overline{L}_{j}=i\mathbb{R}^{+}$
preserving the orientation.
\end{proof}
\begin{lem}
\label{lem:intersect_L1}Let $\gamma=\gamma\left(\xi,\eta\right)$
be a geodesic with $\xi,\eta\in\mathbb{R}$. Then $\gamma$ intersects
the vertical arc $L_{1}$ if and only if $\eta<\frac{\lambda}{2}<\xi$
and $\delta\left(\xi,\eta\right)<0$ where\[
\delta\left(\xi,\eta\right)=\eta-\frac{\xi\lambda-2}{2\xi-\lambda}.\]
For even $q$ we have in particular $\delta\left(\xi,\eta\right)=\eta-\left(TS\right)^{h+1}\xi.$
\end{lem}
\begin{proof}
It is clear, that $\gamma$ intersects $L_{1}$ if and only if the
intersection with $\overline{L}_{1}=\frac{\lambda}{2}+i\mathbb{\mathbb{R}}^{+}$
is at a height above $\sin\frac{\pi}{q}=\Im\rho$. By Lemma \ref{lem:intersection_vertical_and_circular}
the point of intersection is given by $w\left(\xi,\eta\right)=\frac{\lambda}{2}+i\sqrt{\left(\xi-\frac{\lambda}{2}\right)\left(-\eta+\frac{\lambda}{2}\right)}$.
We thus need to check the inequality $\left(\xi-\frac{\lambda}{2}\right)\left(-\eta+\frac{\lambda}{2}\right)>\sin^{2}\frac{\pi}{q}$.
With $\eta<\frac{\lambda}{2}<\xi$ it is clear that $\Im w$ decreases
as $\eta$ increases for $\xi$ fixed. Using $\frac{\lambda}{2}=\cos\frac{\pi}{q}$
we see, that \begin{align*}
\left(\xi-\frac{\lambda}{2}\right)\left(\frac{\lambda}{2}-\eta\right) & =\sin^{2}\frac{\pi}{q}=1-\frac{\lambda^{2}}{4}\Leftrightarrow\\
\lambda-2\eta & =\frac{4-\lambda^{2}}{2\xi-\lambda}\Leftrightarrow\eta=\frac{\lambda\xi-2}{2\xi-\lambda}=A\xi\end{align*}
where $A=\frac{1}{4-\lambda^{2}}\left(\begin{smallmatrix}\lambda & -2\\
2 & -\lambda\end{smallmatrix}\right)\in\SLR.$ Hence $\Im w>\sin\frac{\pi}{q}$ $\Leftrightarrow$ $\eta<A\xi$
$\Leftrightarrow$ $\delta\left(\xi,\eta\right)<0$. 

Observe, that $A\rho=\rho$ and $A^{2}=Id$, i.e. $A$ is elliptic
of order $2$. The stabilizer $G_{q,\rho}$ of $\rho$ in $G_{q}$
is a cyclic group with $q$ elements generated by $TS$. For even
$q=2h+2$ one can use the explicit formula (\ref{eq:TS-explicit})
to verify that $A=\left(TS\right)^{h+1}\in G_{q}$. For odd $q$ on
the other hand there is no element of order $2$ in $G_{q,\rho}$,
so $A\notin G_{q}$. 
\end{proof}
\begin{cor}
\label{cor:intersects_L1_translate}Let $\gamma=\gamma\left(\xi,\eta\right)$
be a geodesic with $\xi,\eta\in\mathbb{R}$. Set $\delta_{n}\left(\xi,\eta\right):=\delta\left(\xi-n\lambda,\eta-n\lambda\right)$.
Then $\gamma$ intersects the line $T^{n}\overline{L}_{1}$ if and
only if $\eta<\left(n+\frac{1}{2}\right)\lambda<\xi$ and $\delta_{n}\left(\xi,\eta\right)<0$. 
\end{cor}

\subsection{\label{sub:First-return-map}The first return map}

Our aim in this section is to obtain an explicit expression for the
first return map $\RET:\Sigma\rightarrow\Sigma$. The notation is
as in Definition \ref{def:First_return_map}, see also Figure \ref{fig:Return-map}.
The main idea is to use geometric arguments to identify possible sequences
of intersections $\left\{ w_{n}\right\} $ and then use arguments
involving regular and dual regular $\lambda$-fractions to determine
whether a particular $w_{n}$ corresponds to a return to $\Sigma$
or not. 

\begin{lem}
If $\xi=\RS{1;1^{j},a_{j+1},a_{j+2}}$ with $\Cq\left(S\xi\right)=\RS{1^{j+1},a_{j+1},\ldots}\in\ZRNCF$
then $\xi\in\left(\left(TS\right)^{j}\lambda,\left(TS\right)^{j}\frac{3\lambda}{2}\right)$
if $a_{j+1}\le-1$ respectively $\xi\in\left(\left(TS\right)^{j+1}\frac{3\lambda}{2},\left(TS\right)^{j}\lambda\right)$
if $a_{j+1}\ge2$. 
\end{lem}
\begin{proof}
Note that $\xi=\left(TS\right)^{j}\xi'$ where $\Cq\left(\xi'\right)=\RS{1;a_{j+1},a_{j+1},\ldots}.$
If $a_{j+1}\le-1$ then $\xi'\in\left(\lambda,\frac{3\lambda}{2}\right)$
and since $TSx=\lambda-\frac{1}{x}$ is strictly increasing there
$\xi=\left(TS\right)^{j}\xi'\in\left(TS\right)^{j}\left(\lambda,\frac{3\lambda}{2}\right)=\left(\left(TS\right)^{j}\lambda,\left(TS\right)^{j}\frac{3\lambda}{2}\right)$.
If on the other hand $a_{j+1}\ge2$ then $\xi''=ST^{-1}\xi'\in\left(\frac{3\lambda}{2},\infty\right)$
and $\xi=\left(TS\right)^{j}\xi'=\left(TS\right)^{j+1}\xi''$ and
therefore $\xi\in\left(TS\right)^{j+1}\left(\frac{3\lambda}{2},\infty\right)=\left(\left(TS\right)^{j+1}\frac{3\lambda}{2},\left(TS\right)^{j}\lambda\right)$.
\end{proof}
\begin{defn}
Define the geodesic arcs \[
\chi_{j}:=\left(TS\right)^{j}L_{2}=\left(TS\right)^{j+1}TL_{-1},\,0\le j\le h,\quad\omega_{j}:=\left(TS\right)^{j}L_{3}=\left(TS\right)^{j+1}L_{0},\,0\le j\le h+1\]
and set $\alpha_{j}:=\left(TS\right)^{j}\lambda$, $\beta_{j}:=\left(TS\right)^{j}\frac{3\lambda}{2}$
and $\delta_{j}:=\left(TS\right)^{j}\left(\lambda+1\right)$. Then
$\chi_{j}=\left[\rho,\alpha_{j}\right],$ $\omega_{j}=\left[\rho,\left(TS\right)^{j}\left(\rho+\lambda\right)\right]\subset\left[\rho,\delta_{j}\right]$
and $\alpha_{j}<\beta_{j}<\delta_{j}$, $0\le j\le h+1$. Note that
$\alpha_{h}=\frac{\lambda}{2}$ and $\chi_{h}=L_{1}'$ for even $q$
while $\delta_{h+1}=\frac{\lambda}{2}$ and $\omega_{h+1}\subseteq L_{1}'$
for odd $q$ (see Figure \ref{fig:Geodesics-from-rho}).%
\begin{figure}
\caption{\label{fig:Geodesics-from-rho}Geodesics through the point $\rho$
for $q=7$ and $8$ ($h=2$ and $3$)}

\includegraphics[scale=0.5]{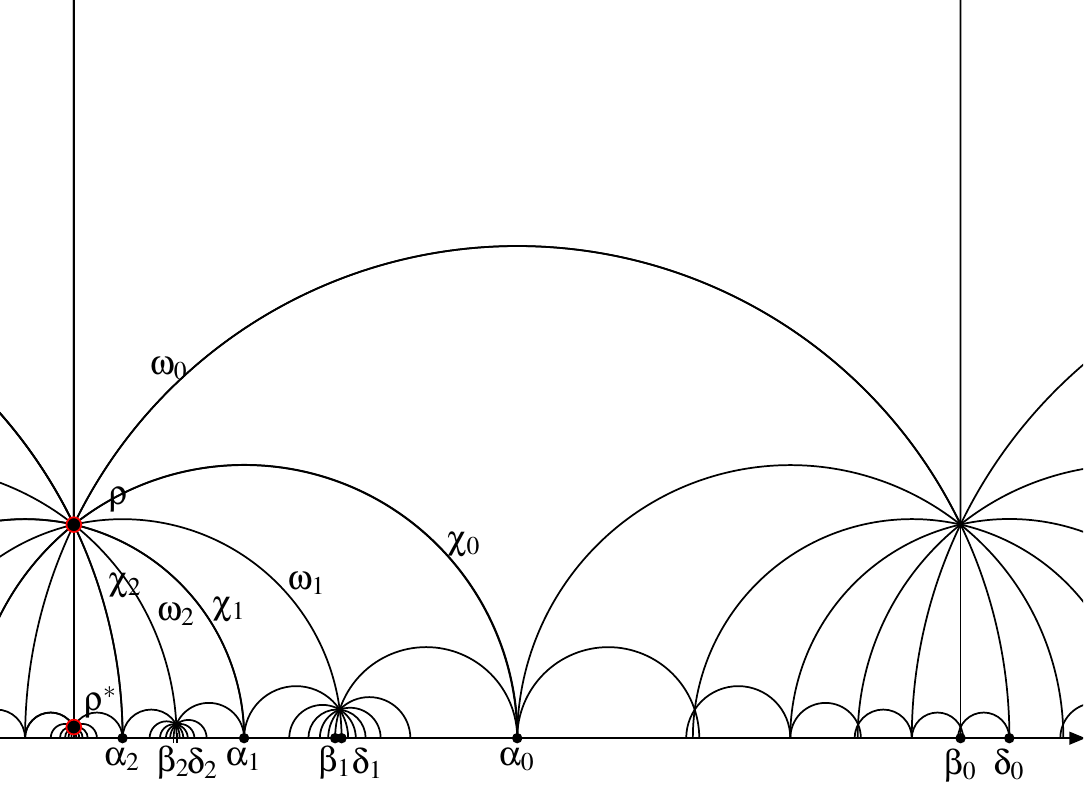}~~\hfill~\includegraphics[scale=0.5]{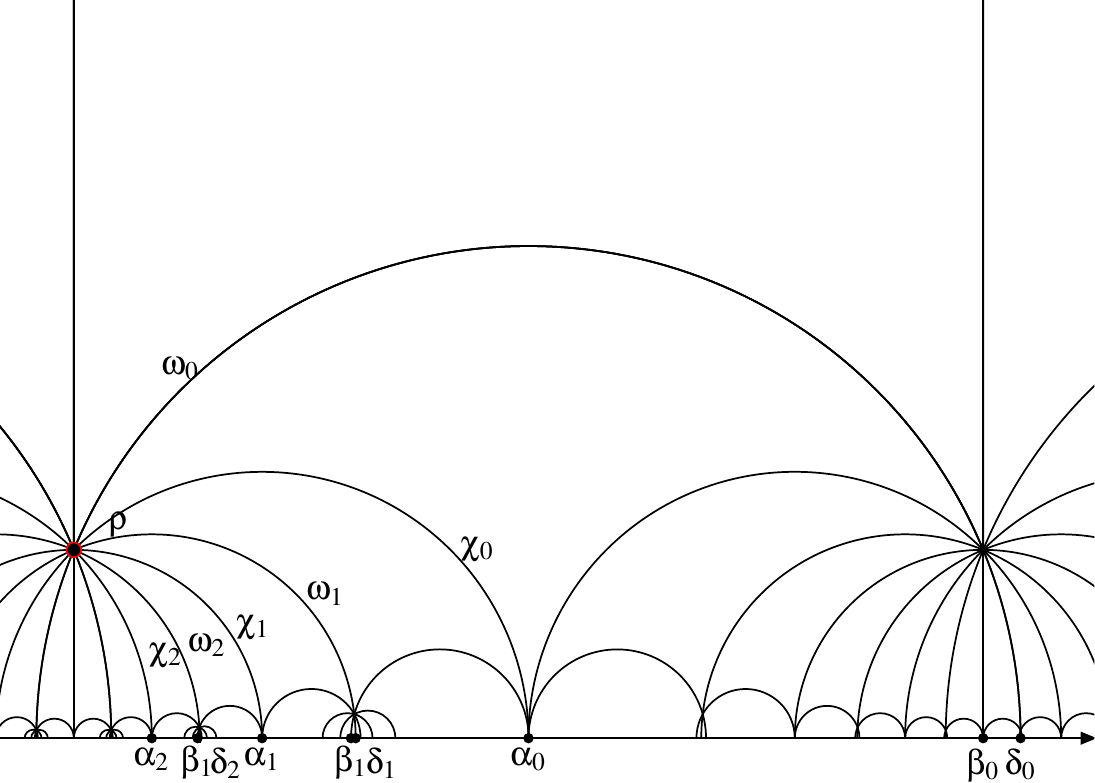}
\end{figure}
 
\end{defn}

\begin{lem}
\label{lem:Following_intersections}If $\gamma=\gamma\left(\xi,\eta\right)\in\Rgeo$
with $\Cq\left(\xi\right)=\RS{1;1^{j},a_{j+1},\ldots}$ then $\gamma$
has the following sequence of intersections with $G_{q}\partial\mathcal{F}$
after passing $L_{1}$: $\omega_{0},\chi_{0},\ldots,\omega_{j-1},\chi_{j-1},\omega_{j}$
if $\xi\in\left(\alpha_{j},\beta_{j}\right)$ ($a_{j+1}\le-1$) respectively
$\omega_{0},\chi_{0},\ldots,\omega_{j},\chi_{j}$ if $\xi\in\left(\beta_{j+1},\alpha_{j}\right)$
($a_{j+1}\ge2$).
\end{lem}
\begin{proof}
See Figure \ref{fig:Geodesics-from-rho}. Since all arcs involved
are hyperbolic geodesics it is clear that $\gamma$ does not intersect
any other $\chi_{i}$'s or $\omega_{i}$'s than those mentioned. Suppose
that $\xi\in\left(\alpha_{j},\beta_{j}\right),$  then after $\chi_{j-1}$
the geodesic $\gamma$ may intersect either $\omega_{j}$ or its extension,
i.e. $\left[\left(TS\right)^{j}\left(\rho+\lambda\right),\delta_{j}\right]$.
If it intersects this extension it has to pass first through the arc
$\left[\left(TS\right)^{j}T\rho,\alpha_{j-1}\right]$. But the completion
of this arc is clearly $\left(TS\right)^{j}TL_{1}=\left[\beta_{j},\alpha_{j-1}\right]$
and hence $\gamma$ can not intersect this arc and must pass through
$\omega_{j}.$ The second case is analogous, except that we do not
care about whether the next intersection is at $\omega_{j+1}$ or
$\left(TS\right)^{j+1}TL_{1}$. 
\end{proof}
\begin{lem}
\label{lem:first_return_map_as_Fqx^j+1_as_a0=00003D1}If $\gamma\left(\xi,\eta\right)\in\Sgeo$
and $\Cq\left(\xi\right)=\RS{1;1^{j},a_{j+1},\ldots}$ then $\RETx\left(\xi,\eta\right)=\Fqx^{j+1}\left(\xi,\eta\right).$ 
\end{lem}
\begin{proof}
Let $\left(z,\theta\right)=\PT\left(\xi,\eta\right),$ then $z\in L_{-1}\cup L_{0}$
and by Lemma \ref{lem:Following_intersections} the subsequent intersections
are $\left\{ w_{0},w_{1},w_{2},\ldots,w_{2j+1},w_{2j+2}\right\} $
if $a_{j+1}\le-1$ and $\left\{ w_{0},w_{1},w_{2},\ldots,w_{2j+2},w_{2j+3}\right\} $
if $a_{j+1}\ge2$. It is also easy to verify, that the corresponding
maps are $A_{2i+1}=T^{-1}\left(ST^{-1}\right)^{i}$ and $A_{2i}=\left(ST^{-1}\right)^{i}$.
Note that $A_{2i+1}=T^{-1}A_{2i}$ and that $A_{2i}\gamma=\Fqx\,^{i}\gamma,\,0\le i\le j+1$.
It is thus clear, that $A_{2i}\gamma\in\Rgeo$ and $A_{2i+1}\gamma\notin\Rgeo$
for $1\le i\le j+1.$ If $\left(\xi_{i},\eta_{i}\right)=\left(A_{i}\xi,A_{i}\eta\right),$
then $\Cq\left(\xi_{2i}\right)=\RS{1;1^{j-i},a_{j+1},\ldots}$ and
$\Cd\left(\eta_{2i}\right)=\DS{\left(-1\right)^{i},b_{1},\ldots}$
and hence $A_{2i}\gamma\notin\Sgeo$ for $1\le i\le j$ but $A_{2j+2}\gamma\in\Sgeo.$
Thus in both cases the first return is given by $w_{2j+2}$ and the
return map is $\RETx=\Fqx\,^{j+1}.$
\end{proof}
\begin{defn}
\label{def:Define_K}Define $\K:\mathbb{R\rightarrow\mathbb{N}}$
and $\n:\mathbb{R\rightarrow\mathbb{Z}}$ as follows: if $\xi=\RS{a_{0};\left(\epsilon\right)^{k-1},a_{k},\ldots}$
with $k\ge1$, and $\epsilon=\sign a_{0}$, then $\K\left(\xi\right):=k$
and \[
\n\left(\xi\right):=\begin{cases}
\epsilon\cdot3, & k=h+1,\, q\,\mbox{odd},\\
\epsilon\cdot2, & k=h,\, a_{h}\ge2,\, q\,\mbox{even},\\
\epsilon\cdot1, & k=h+1,\, q\,\mbox{even},\\
0, & \mbox{else}.\end{cases}\]

\end{defn}
We also have to consider the return map for the second type of strongly
reduced geodesics. 

\begin{lem}
\label{lem:first_return_map_as_Fqxk}For $\mathbf{z}\in\Sigma$ with
$\PT^{-1}\mathbf{z}=\left(\xi,\eta\right)\in\Sred$ and $\left|\xi\right|>\frac{3\lambda}{2}$
one has $\RET\mathbf{z}=\PP\circ\Fqx^{k}\circ\PP^{-1}\mathbf{z}\in\Sigma^{n}$
where $k=\K\left(\xi\right)$ and $n=\n\left(\xi\right)$. 
\end{lem}
\begin{proof}
Consider $\mathbf{z}_{0}=\PP\gamma\in$$\Sigma$ with $\gamma=\gamma\left(\xi,\eta\right)\in\Sgeo$
and assume without loss of generality that $\xi>0$ with $\Cq\left(\xi\right)=\RS{a_{0};1^{j},a_{j+1},\ldots}$
for some $j\ge0$, $a_{j+1}\ne1$ and $a_{j+1}\ne\pm1$ if $j=0$
(the case of $-1$'s is analogous). Recall the notation in Definition
\ref{def:First_return_map}, in particular the sequence $\left\{ w_{n}\right\} _{n\in\mathbb{Z}}$
and the corresponding maps $A_{n}\in G_{q}$. It is clear, that $w_{n}$
gives a return if and only if $A_{n}\gamma\in\Sgeo$.

There are two cases to consider: Either $\mathbf{z}_{0}\in\Sigma^{-1}\cup\Sigma^{0}$
respectively $\mathbf{z}_{0}\in\Sigma^{-1}\cup\Sigma^{0}\cup\Sigma^{3}$
in the case of odd $q$ or $\mathbf{z}_{0}\in\Sigma^{2}$. In Figure
\ref{fig:Geodesics-leaving} these different possibilities are displayed,
$\PP\gamma_{A}\in\Sigma^{-1}$, $\PP\gamma_{B}\in\Sigma^{0},$$\PP\gamma_{C}\in\Sigma^{2}$
and $\PP\gamma_{D}\in\Sigma^{3}$. It is clear, that if $\mathbf{z}_{0}\in\Sigma^{2}$
then $w_{0}=TSz_{0}\in L_{2}$ and the sequence of $\left\{ w_{n}\right\} $
is essentially different from the case $\mathbf{z}_{0}\not\in\Sigma^{2}$
when $w_{0}=z_{0}$.%
\begin{figure}
\caption{\label{fig:Geodesics-leaving}Geodesics leaving the Poincaré section
($q$=7)}
\includegraphics[scale=0.5]{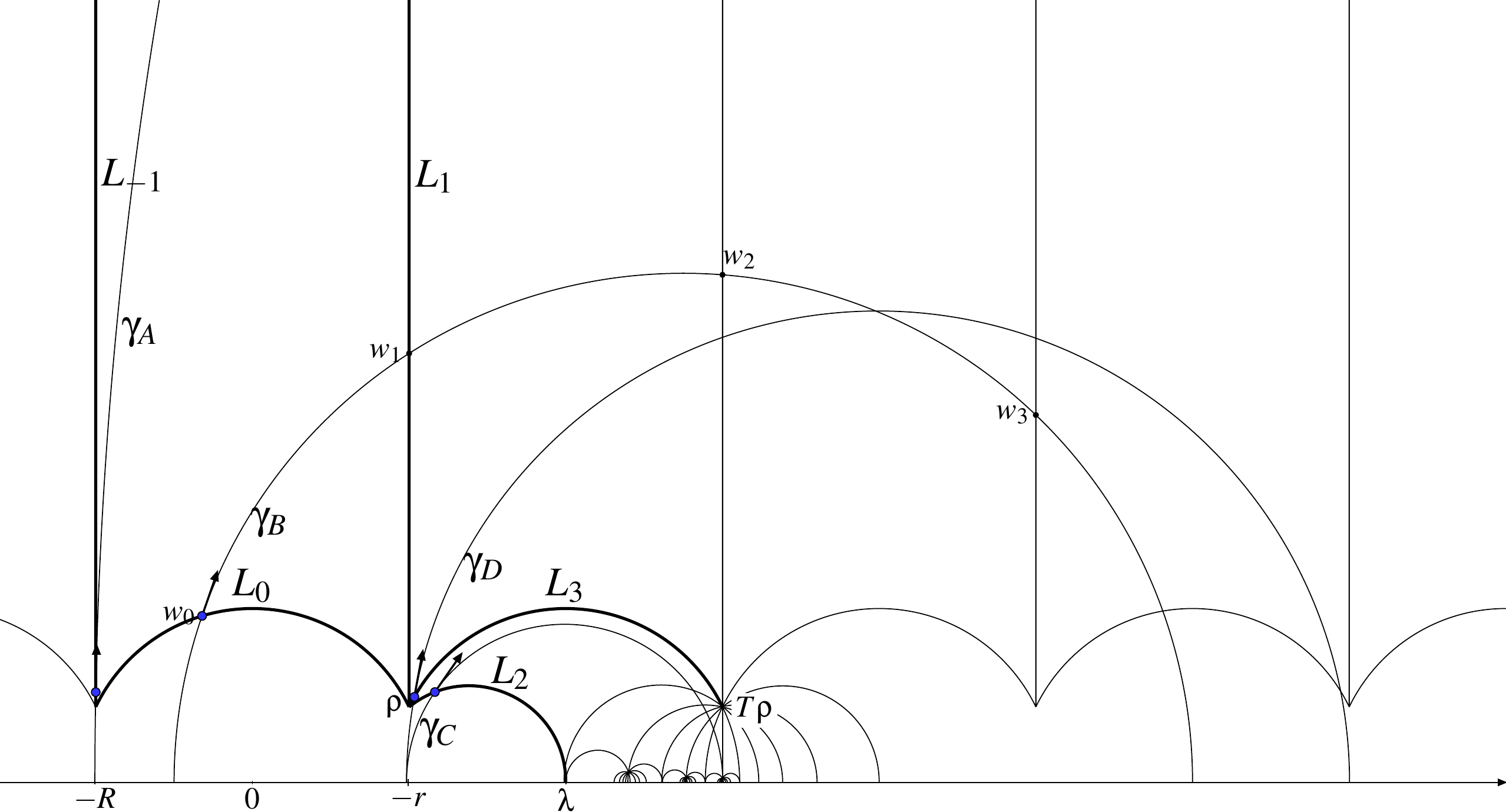}
\end{figure}
\begin{figure}
\caption{\label{fig:Geodesics-returning-even}Geodesics returning to the Poincaré
section ($q=8$)}
\includegraphics[scale=0.5]{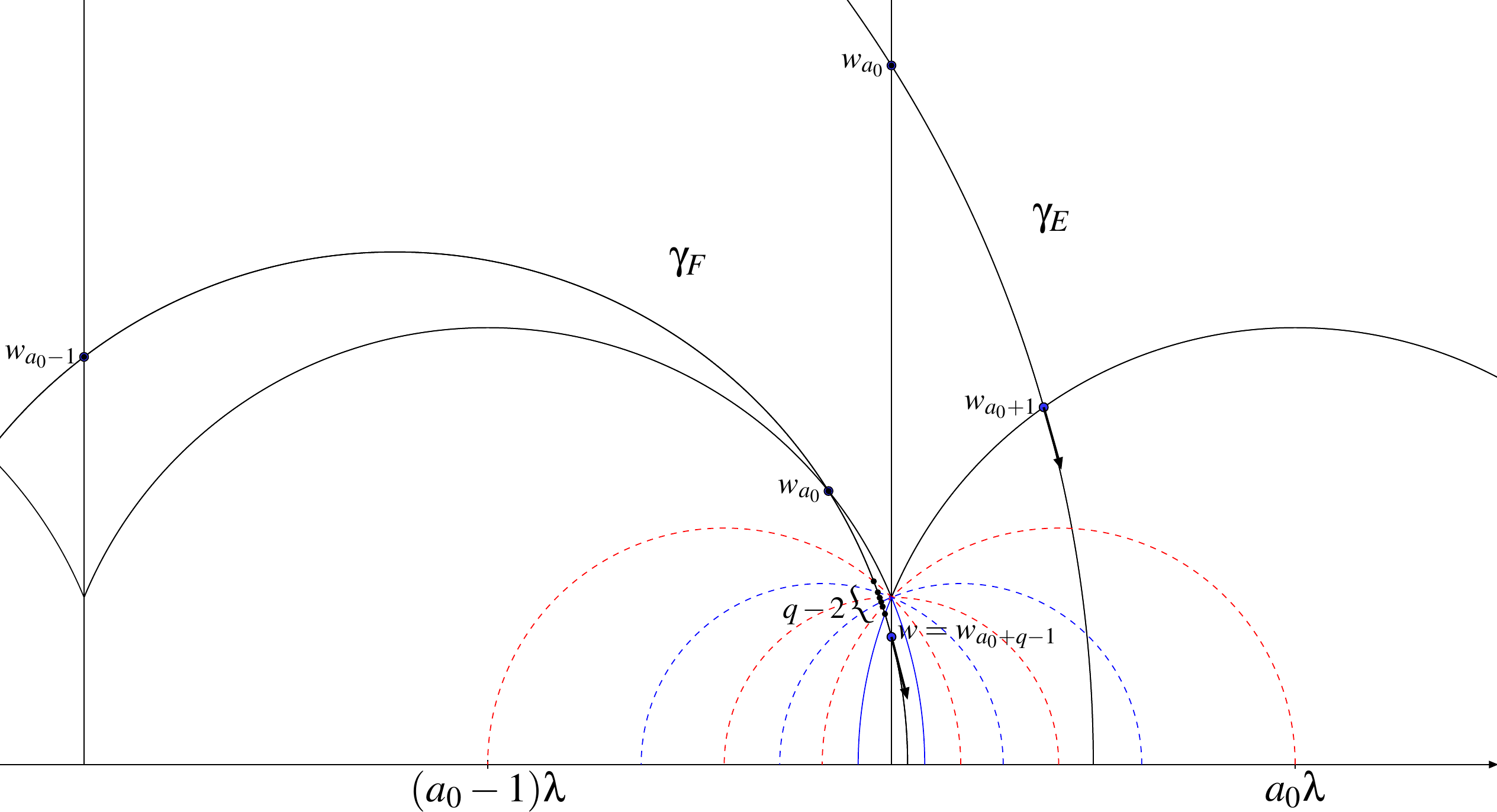}
\end{figure}

\textbf{Case 1:} If $\mathbf{z}_{0}\not\in\Sigma^{2}$ (see geodesics
$\gamma_{A}$, $\gamma_{B}$ and $\gamma_{D}$ in Figure \ref{fig:Geodesics-leaving}),
then $w_{n}\in T^{n}L_{-1}$ for $1\le n\le k-1$ and $k=a_{0}-1,a_{0}$
or $a_{0}+1$ depending on whether $\mathbf{z}_{0}\in\Sigma^{0,1}$
or $\Sigma^{3}$ and whether the next intersection is on $T^{a_{0}}L_{0}$
or $T^{a_{0}-1}L_{0}$. Then either $w_{k}\in T^{a_{0}}L_{0}$ or
$w_{k}\in T^{a_{0}-1}L_{0}$ (see geodesics $\gamma_{E}$ and $\gamma_{F}$
in Figure \ref{fig:Geodesics-returning-even}). Since $A_{n}=T^{-n}$
for $w_{n}\in T^{n}L_{-1}$ and, as we will show in Lemma \ref{lem:translates_not_reduced}
$T^{-n}\gamma\notin\Rgeo$ none of the $w_{n}\in T^{n}L_{-1}$ for
$1\le n\le k-1$ gives a return to $\Sigma$. There are now two possibilities:
\begin{lyxlist}{00.00.0000}
\item [{(i)}] If $w_{k}\in T^{a_{0}}L_{0}$, then $A_{k}=ST^{-a_{0}}$
and $\gamma_{k}=A_{k}\gamma=\Fqx\gamma$. If $j=0$ it is clear that
$\gamma'\in\Sgeo$ and $\RET\mathbf{z}_{0}=\PP\circ\Fqx\gamma\in\Sigma^{0}$.
If $j\ge1$, by Lemma \ref{lem:first_return_map_as_Fqx^j+1_as_a0=00003D1}
applied to $\gamma'$ we get $\RET\mathbf{z}_{0}=\PP\circ\Fqx^{j}\gamma'=\PP\circ\Fqx^{j+1}\gamma\in\Sigma^{\n\left(\xi\right)}$.

\item [{(ii)}] If $w_{k}\in T^{a_{0}-1}L_{0}$, then we will show in Lemma
\ref{lem:even_crossing_dj_gives_no_return} and \ref{lem:odd_crossing_dj_gives_no_return}
that none of the arcs emanating from $T^{a_{0}-1}\rho$ gives a return
(cf.~Figures \ref{fig:Geodesics-returning-even} and \ref{fig:Geodesics-from-rho})
except for the next return at $T^{a_{0}-1}\left(ST^{-1}\right)^{h+1}L_{1}$.
Furthermore it follows, that $\RET\mathbf{z}_{0}=\PP\circ\Fqx^{k}\circ\PP^{-1}\mathbf{z}_{0}\in\Sigma^{n}$
with $k=\K\left(\xi\right)$ (here $h$ or $h+1$) and $n=\n\left(\xi\right)$. 
\end{lyxlist}
\textbf{Case 2: }If $\mathbf{z}_{0}\in\Sigma^{2}$ then $\frac{3\lambda}{2}<\xi<\lambda+1$
and $\gamma$ must intersect $T\overline{L}_{1}$ below $T\rho$.
By the same arguments as in \textbf{Case 1} we conclude that the first
return is given by $w_{q-1}\in T\left(ST^{-1}\right)^{h+1}L_{1}$
and $\RET\mathbf{z}_{0}=\PP\circ\Fqx^{k}\circ\PP^{-1}\mathbf{z}_{0}\in\Sigma^{n}$
where $k=\K\left(\xi\right)$ and $n=\n\left(\xi\right)$ as in \textbf{Case
1} (ii). In all cases we see, that the first return map $\RET:\Sigma\rightarrow\Sigma$
is given by $\RET=\PP\circ\Fqx^{k}\circ\PP^{-1}$ or alternatively
by $\RETx=\Fqx^{k}$ where $k=1,$ $h$ or $h+1$ depending on $\xi$. 
\end{proof}
By combining Lemma \ref{lem:first_return_map_as_Fqx^j+1_as_a0=00003D1}
and \ref{lem:first_return_map_as_Fqxk} it is easy to see, that the
first return map $\PT$ is determined completely in terms of the coordinate
$\xi$: 

\begin{prop}
\label{pro:even_The-return-map-is-Fxk-in_domain}If $\mathbf{z}\in\Sigma$
with $\PT^{-1}\mathbf{z}=\left(\xi,\eta\right)\in\Sred$ then $\RET\mathbf{z}=\PT\circ\Fqx\,^{k}\,\circ\PT^{-1}\mathbf{z}\in\Sigma^{n}$
where $k=\K\left(\xi\right)$ and $n=\n\left(\xi\right)$.  
\end{prop}
Having derived explicit expressions for the first return map, in a
next step we want to get explicit formulas for the first return time,
i.e. the hyperbolic length between the successive returns to $\Sigma$.

\subsection{The first return time }

\begin{lem}
\label{lem:d=00003Dlng0-lng1+2lnF}Let $\gamma=\gamma\left(\xi,\eta\right)\in\Sgeo$
with $\xi=\RS{a_{0};\left(\epsilon\right)^{k-1},a_{k},\ldots}$ ($\epsilon=\sign a_{0}$)
and $\PP\gamma=\mathbf{z}_{0}=\left(z_{0},\theta_{0}\right)$ and
let $\mathbf{z}_{1}=\RET\mathbf{z}_{0}$. For $w_{0}\in\gamma$ the
point corresponding to $\mathbf{z}_{1}$, i.e. $w_{0}\in G_{q}z_{1},$
one has \[
d\left(z_{0},w_{0}\right)=\ln g\left(z_{0},\gamma\right)-\ln g\left(z_{1},\RETx\gamma\right)+2\ln F\left(\gamma\right)\]
where $F\left(\gamma\right)=\prod_{j=1}^{k}\left|\xi_{j}\right|$
with $k=\K\left(\xi\right)$ as in Definition \ref{def:Define_K}
and $\xi_{j}=S\Fq^{j}S\xi$. 
\end{lem}
\begin{proof}
Set $\gamma_{j}:=\Fqx^{j}\gamma=B_{j}\gamma$ and $w_{j}:=B_{j}w_{0}$.
By Lemma \ref{lem:g(Az,Aw)=00003D} $g\left(w_{0},\gamma\right)=g\left(B_{1}^{-1}w_{1},B_{1}^{-1}\gamma_{1}\right)=g\left(w_{1},\gamma_{1}\right)\xi_{1}^{-2}$.
Applying the same formula to $g\left(w_{j},\gamma_{j}\right)$ for
$j=1,\ldots,k$ we get $\ln g\left(w_{0},\gamma\right)=\ln g\left(z_{1},\RETx\gamma\right)-2\ln\prod_{j=1}^{k}\left|\xi_{j}\right|$.
The statement then follows by Proposition \ref{pro:even_The-return-map-is-Fxk-in_domain}
and Lemma \ref{lem:hyp_dist_points_on_geod}.
\end{proof}
\begin{lem}
If $\gamma=\gamma\left(\xi,\eta\right)$ is a strongly reduced closed
geodesic with $\Cq\left(S\xi\right)=\RS{\overline{a_{1},\ldots,a_{n}}}$
of minimal period $n$ and $\xi_{j}=S\Fq^{j}S\xi$, then the hyperbolic
length of $\gamma$ is given by \begin{equation}
l\left(\gamma\right)=2\sum_{j=1}^{n}\ln\left|\xi_{j}\right|=-\ln\prod_{j=1}^{n}\left|\RS{\overline{a_{j+1},\ldots,a_{n},a_{1},\ldots,a_{j}}}\right|^{2}.\label{eq:length_of_closed_geodesic}\end{equation}

\end{lem}
\begin{proof}
Denote by $\left(z_{j},\theta_{j}\right)\in\Sigma$ the successive
returns of $\gamma$ to $\Sigma$ and let $w_{j-1}\in\gamma$ be the
point on $\gamma$ corresponding to $z_{j}$. If $\gamma$ is closed,
the set $\left\{ z_{j}\right\} _{j\ge0}$ is finite with $N+1$ elements
for some $N+1\le n$, i.e. $z_{N+1}=z_{0}$. It is clear that the
length of $\gamma$ is given by adding up the lengths of all pieces
between the successive returns to $\Sigma$ and a repeated application
of Lemma \ref{lem:d=00003Dlng0-lng1+2lnF} gives us \begin{align*}
l\left(\gamma\right) & =\sum_{j=0}^{N}d\left(z_{j},w_{j}\right)=\sum_{j=0}^{N}\left(\ln g\left(z_{j},\RETx^{j}\gamma\right)-\ln g\left(w_{j},\RETx^{j}\gamma\right)\right)\\
 & =\sum_{j=0}^{N}\left(\ln g\left(z_{j},\RETx^{j}\gamma\right)-\ln g\left(z_{j+1},\RETx^{j+1}\gamma\right)+2\ln F\left(\RETx^{j}\gamma\right)\right)\\
 & =2\sum_{j=0}^{N}\ln F\left(\RETx^{j}\gamma\right)=2\ln\prod_{i=1}^{n}\left|S\Fq^{i}S\xi\right|.\end{align*}

\end{proof}
\begin{rem}
Formula (\ref{eq:length_of_closed_geodesic}) can also be obtained
by relating the length of $\gamma$$\left(\xi,\eta\right)$ to the
axis of the hyperbolic matrix fixing $\xi$ and observing that this
matrix must be given by the map $\Fq^{n}$ acting on $\xi$.

In the case of $\PSLZ$ and the Gauss (regular) continued fractions
formula (\ref{eq:length_of_closed_geodesic}) is well-known.
\end{rem}

We are now in a position to discuss the first return time. By Lemma
\ref{lem:hyp_dist_points_on_geod} it is clear that we need to calculate
the function $g_{j}\left(\xi,\eta\right)=\frac{\left|w_{j}-\xi\right|^{2}}{v_{j}},$
where $w_{j}=u_{j}+iv_{j}=\mathcal{Z}_{j}\left(\xi,\eta\right)$ for
all the intersection points in Corollary \ref{cor:_explicit_formulas_for_g023}.

\begin{defn}
\label{def:Xi_k}Let $B\in\PSLR$ be given by $Bz=\frac{2-\lambda z}{\lambda-2z}$.
Set $\delta_{n}\left(\xi,\eta\right):=\eta-T^{n}BT^{-n}\xi$ and $\Xi_{+}:=\left\{ \left(\xi,\eta\right)\in\Sred\,\bigg{|}\,\delta_{\nlm{x}-1}\left(\xi,\eta\right)\ge0\right\} $,
$\Xi_{-}:=\left\{ \left(\xi,\eta\right)\in\Sred\,\bigg{|}\,\delta_{\nlm{x}-1}\left(\xi,\eta\right)<0\right\} $. 
\end{defn}
\begin{prop}
\label{pro:return_time}The first return time $r$ for the geodesic
$\gamma\left(\xi,\eta\right)$ is given by the function \begin{align*}
r\left(\xi,\eta\right) & =\ln g_{A}\left(\xi,\eta\right)-\ln g_{\n\left(\xi\right)}\left(\Fqx^{\K\left(\xi\right)}\left(\xi,\eta\right)\right)+2\ln F\gamma\left(\xi,\eta\right)\,\,\mbox{for}\,\,\xi>0,\\
\text{respectively}\\
r\left(\xi,\eta\right) & =r\left(-\xi,-\eta\right),\,\mbox{for}\,\xi<0.\end{align*}
Thereby\[
A=A\left(\xi,\eta\right)=\begin{cases}
-1, & -R\le\eta<-\frac{\lambda}{2},\,\xi>-B\left(-\eta\right),\\
0, & -R\le\eta<-\frac{\lambda}{2},\,\xi<-B\left(-\eta\right),\,\mbox{or}-\frac{\lambda}{2}\le\eta<-r,\,\xi\ge B\left(\eta\right),\\
2, & \frac{3}{4}\lambda-\frac{1}{\lambda}<\eta<-r,\,\frac{3\lambda}{2}<\xi<B\left(\eta\right)<\lambda+1,\\
3, & \lambda-1<\eta<-r,\,\lambda+1<\xi<B\left(\eta\right),\end{cases}\]
$\K\left(\xi\right)$ and $\n\left(\xi\right)$ are defined as in
Definition \ref{def:Define_K}, whereas the functions $g_{j}\left(\xi,\eta\right)=g\left(z_{j},\gamma\right)$
for $z_{j}\in L_{j}$ are given as in Corollary \ref{cor:_explicit_formulas_for_g023}
and $F\left(\gamma\right)$ is given as in Lemma \ref{lem:d=00003Dlng0-lng1+2lnF}. 
\end{prop}
\begin{proof}
Consider $\gamma=\gamma\left(\xi,\eta\right)\in\Sgeo$ with $\xi>0$
and suppose that $\PP\gamma=\mathbf{z}_{0}\in\Sigma$ and $\RET\mathbf{z}_{0}=\mathbf{z}_{1}\in\Sigma$
with $w\in\gamma$ corresponding to $z_{1}.$ Since geodesics are
parametrized by arc length the first return time is simply the hyperbolic
length between $z_{0}$ and $w$, i.e. \[
r\left(\xi,\eta\right)=d\left(z_{0},w\right)=\ln g\left(z_{0},\gamma\right)-\ln g\left(w,\gamma\right)=\ln g\left(z_{0},\gamma\right)-\ln g\left(z_{1},\Fqx\gamma\right)+2\ln F\left(\gamma\right)\]
by Lemma \ref{lem:d=00003Dlng0-lng1+2lnF}. If $z_{0}\in L_{j}$,
we set $g\left(z_{0},\gamma\right)=g_{j}\left(\xi,\eta\right)$ as
given in Corollary \ref{cor:_explicit_formulas_for_g023}. By Corollary
\ref{cor:intersects_L1_translate} it is easy to verify, that the
sets in the definition of $A\left(\xi,\eta\right)$ correspond exactly
to the cases $z_{0}\in L_{-1},L_{0},L_{2}$ and $L_{3}$ respectively,
where the last set is empty for even $q$. It is also easy to see,
that $B\left(\frac{3}{4}\lambda-\frac{1}{\lambda}\right)=\frac{3\lambda}{2}$
and $B\left(\lambda-1\right)=\lambda+1$. The statement of the Proposition
now follows from the explicit formula for $F\left(\gamma\right)$
in Lemma \ref{lem:d=00003Dlng0-lng1+2lnF} and the domains in Proposition
\ref{pro:even_The-return-map-is-Fxk-in_domain} for which $\RETx=\Fqx^{k}$.
That $r\left(-\xi,-\eta\right)=r\left(\xi,\eta\right)$ follows from
the invariance of the cross-section with respect to reflection in
the imaginary axis. 
\end{proof}

\section{\label{sec:Applications}Construction of an Invariant Measure }

By Liouvilles theorem we know that the geodesic flow on $\SH$ preserves
the measure induced by the hyperbolic metric. This measure, the \emph{Liouville
measure, }is given by $dm=y^{-2}dxdyd\theta$ in the coordinates $\left(x+iy,\theta\right)\in\H\times S^{1}$
on $\SH$. Using the coordinates $\left(\xi,\eta,s\right)\in\Omega^{*}\times\mathbb{R}$
given by Corollary \ref{cor:change_var_x-y-theta->xi-eta-s} we obtain
the Liouville measure in these coordinates\[
dm=\frac{dxdyd\theta}{y^{2}}=\left|\frac{\partial\left(x,y,\theta\right)}{\partial\left(\xi,\eta,s\right)}\right|\frac{2}{r^{2}\cos^{2}\theta}d\xi d\eta ds=\frac{2d\xi d\eta dt}{\left(\eta-\xi\right)^{2}}.\]
The time-discretization of the geodesic flow in terms of the cross-section
and first return map thus preserves the measure $dm'=\frac{2d\xi d\eta}{\left(\eta-\xi\right)^{2}}$.
We prefer to work with the finite domain $\Sred\subseteq\Omega^{\infty}$.
So let $u=S\xi$ and $v=-\eta$ with $\left(u,v\right)\in\Omega^{\infty}$.
Hence the measure $d\mu\left(u,v\right)=dm'\left(\xi,\eta\right)$
given by \[
d\mu=\frac{2dudv}{\left(1-uv\right)^{2}}\]
on $\Omega^{\infty}$ is invariant under $\RETx\left(u,v\right)=\Fqx\,^{\K\left(u\right)}\left(u,v\right)=\left(f_{1}\left(u\right),f_{2}\left(v\right)\right)$
where $f_{1}\left(u\right)=\Fq^{\K\left(u\right)}u$ $=Au$ and $f_{2}\left(v\right)=-A\left(-v\right)$.
Because $d\mu$ is equivalent to Lebesgue measure, we deduce that
$d\mu$ is in fact an $\Fq^{\K\left(u\right)}$ invariant measure
on $\Omega$. If $\pi_{x}\left(x,y\right)=x$ it is clear that $\pi_{x}\circ\RETx\left(u,v\right)=f_{1}\left(u\right)=f_{1}\circ\pi_{x}\left(u,v\right)$
so $f_{1}$ is a factor map of $\RETx$. An invariant measure of $f_{1}:I_{q}\rightarrow I_{q}$
can be obtain by integrating $d\mu$ in the $v$-direction. We get
different alternatives depending on $q$ being even or odd.

\subsection{$q=3$}

In this case the set of strongly reduced and reduced geodesics are
the same and $\Fq^{\K}=\Fq$ if we set $\mathcal{U}_{1}=\left[-\frac{1}{2},0\right]=-\mathcal{U}_{-1}$
and $\mathcal{V}_{1}=\left[r,R\right]$ then $d\mu=\chi_{\mathcal{U}_{1}}d\mu_{1}+\chi_{\mathcal{U}_{-1}}d\mu_{-1}$
where \begin{align*}
d\mu_{1}\left(u\right) & =\int_{r}^{R}\frac{2dudv}{\left(1-uv\right)^{2}}=\left[\frac{1}{u\left(1-uv\right)}\right]_{v=r}^{v=R}du\\
 & =\frac{1}{u\left(1-uR\right)}-\frac{1}{u\left(1-ur\right)}du=\frac{1-ur-1+uR}{u\left(1-uR\right)\left(1-ur\right)}du\\
 & =\frac{1}{\left(1-uR\right)\left(1-ur\right)}du\end{align*}
and $d\mu_{-1}\left(u\right)=-d\mu_{1}\left(-u\right)=\frac{1}{\left(1+uR\right)\left(1+ur\right)}du.$
Here $r=\frac{\sqrt{5}-3}{2}$ and $R=\frac{\sqrt{5}-1}{2}$ and $R-r=1$.

\subsection{Even $q\ge4$ }

Here $\mathcal{U}_{1}=\left[-\frac{\lambda}{2},-\frac{2}{3\lambda}\right]$,
$\mathcal{U}_{2}=\left[-\frac{2}{3\lambda},0\right]$, $\mathcal{V}_{1}=\left[0,R\right]$
and $\mathcal{V}_{2}=\left[r,R\right]$. Hence \begin{eqnarray*}
\frac{d\mu_{1}}{du}\left(u\right) & = & \int_{0}^{R}\frac{2dv}{\left(1-uv\right)^{2}}=\frac{1}{u\left(1-uR\right)}-\frac{1}{u}=\frac{R}{1-uR},\\
\frac{d\mu_{2}}{du}\left(u\right) & = & \int_{r}^{R}\frac{2dv}{\left(1-uv\right)^{2}}=\frac{1}{u\left(1-uR\right)}-\frac{1}{u\left(1-ur\right)}\\
 & = & \frac{R-r}{\left(1-uR\right)\left(1-ur\right)}=\frac{\lambda}{\left(1-uR\right)\left(1-ur\right)},\\
\frac{d\mu_{-j}}{du}\left(u\right) & = & \frac{d\mu_{j}}{du}\left(-u\right)\end{eqnarray*}
and the invariant measure of $\Fq^{\K}$ for odd $q$ is given by
\[
d\mu\left(u\right)=\sum_{j=-3}^{3}\chi_{\mathcal{U}_{j}}\left(u\right)d\mu_{j}\left(u\right)\]
where $\chi_{\mathcal{I}_{j}}$ is the characteristic function for
the interval $\mathcal{U}_{j}$. This measure is piece-wise differentiable
and finite. The finite-ness is clear since $uR$ and $ur\ne1$ for
$u\in I_{q}.$ If $\int_{I_{q}}d\mu\left(u\right)=c$ then $\frac{1}{c}d\mu$
is a probability measure on $I_{q}$.

\subsection{Odd $q\ge3$}

Let $\mathcal{U}_{1}=\left[-\frac{\lambda}{2},\frac{-2}{3\lambda}\right]$,
$\mathcal{U}_{2}=\left[-\frac{2}{3\lambda},-\frac{1}{2\lambda}\right]$,
$\mathcal{U}_{3}=\left[-\frac{1}{2\lambda},0\right]$, $\mathcal{V}_{1}=\left[0,R\right],$
$\mathcal{V}_{2}=\left[r_{\kappa-1},R\right]$ and $\mathcal{V}_{3}=\left[r,R\right]$.
Then \begin{align*}
\frac{d\mu_{1}}{du}\left(u\right) & =\int_{0}^{R}\frac{2dv}{\left(1-uv\right)^{2}}=\frac{1}{u\left(1-uR\right)}-\frac{1}{u}=\frac{R}{1-uR},\\
\frac{d\mu_{2}}{du}\left(u\right) & =\int_{r_{\kappa-1}}^{R}\frac{2dv}{\left(1-uv\right)^{2}}=\frac{1}{u\left(1-uR\right)}-\frac{1}{u\left(1-ur_{\kappa-1}\right)}\\
 & =\frac{R-r_{\kappa-1}}{\left(1-uR\right)\left(1-ur_{\kappa-1}\right)},\\
\frac{d\mu_{3}}{du}\left(u\right) & =\frac{R-r}{\left(1-uR\right)\left(1-ur\right)}=\frac{\lambda}{\left(1-uR\right)\left(1-ur\right)},\\
\frac{d\mu_{-j}}{du}\left(u\right) & =\frac{d\mu_{j}}{du}\left(-u\right)\end{align*}

and the invariant measure of $\Fq^{\K}$ for odd $q$ is given by
\[
d\mu\left(u\right)=\sum_{j=-3}^{3}\chi_{\mathcal{U}_{j}}\left(u\right)d\mu_{j}\left(u\right)\]
where $\chi_{\mathcal{I}_{j}}$ is the characteristic function for
the interval $\mathcal{U}_{j}$. This measure is piece-wise differentiable
and finite. The finite-ness is clear since $uR,$ $ur$ and $ur_{\kappa-1}\ne1$
for $u\in I_{q}.$ If $\int_{I_{q}}d\mu\left(u\right)=c$ then $\frac{1}{c}d\mu$
is a probability measure on $I_{q}$. 

\begin{rem}
For another approach leading to an infinite invariant measure see
e.g. Haas and Gröchenig \cite{MR1424398}.
\end{rem}

\subsection{Invariant measure for $\Fq$}

It is easy to verify that $dm\left(\xi,\eta\right)=\frac{2d\xi d\eta}{\left(\xi-\eta\right)^{2}}$
is invariant under Möbius transformations, i.e. if $A\in\PSLR$ then
$dm\left(A\xi,A\eta\right)=dm\left(\xi,\eta\right)$. By considering
the action of $\Fqx$ on $\Omega^{*},$ i.e. \[
\Fqx\left(\xi,\eta\right)=\Sx\circ\Fqx\circ\Sx\left(\xi,\eta\right)=\left(S\Fq S\xi,\frac{1}{n\lambda-\eta}\right)=\left(ST^{-n}\xi,ST^{-n}\eta\right)\]
it is clear that $dm$ is invariant under $\Fqx:\Omega^{*}\rightarrow\Omega^{*}$
and letting $u=S\xi$ and $v=-\eta$ it is easy to verify that $dm'\left(u,v\right)=\frac{2dudv}{\left(1-uv\right)^{2}}$
is invariant under $\Fqx:\Omega\rightarrow\Omega$. We can thus obtain
corresponding invariant measure $d$$\mu$ for $\Fq$ by projecting
on the first variable. Let $d\mu\left(u\right)=d\mu_{j}\left(u\right)$
for $u\in\mathcal{I}_{j},$then \begin{align*}
\frac{d\mu_{j}}{du}\left(u\right) & =2\int_{r_{j}}^{R}\frac{dv}{\left(1-uv\right)^{2}}\\
 & =2\left[\frac{1}{u\left(1-uv\right)}\right]_{r_{j}}^{R}=\frac{2}{u}\left[\frac{1}{1-Ru}-\frac{1}{1-r_{j}u}\right]=\frac{2\left(R-r_{j}\right)}{\left(1-Ru\right)\left(1-r_{j}u\right)}\end{align*}
and the invariant measure of $\Fq$ is given by \[
d\mu\left(u\right)=\sum_{j=-\np}^{\np}\chi_{\mathcal{I}_{j}}\left(u\right)d\mu_{j}\left(u\right)\]
where $\chi_{\mathcal{I}_{j}}$ is the characteristic function for
the interval $\mathcal{I}_{j}$. This measure is piece-wise differentiable
and finite. If $c=\int_{I_{q}}d\mu$ then $\frac{1}{c}d\mu$ is a
probability measure on $I_{q}$. The explicit values of $c=\frac{1}{4}C$
where $C^{-1}=\ln\left(\frac{1+\cos\frac{\pi}{q}}{\sin\frac{\pi}{q}}\right)$
for even $q$ and $C^{-1}=\ln\left(1+R\right)$ for odd $q$ (see
Lemma 3.2 and 3.4 in \cite{MR1650073}).

\section{\label{sec:Lemmas-on-continued}Lemmas on continued fraction expansions
and reduced geodesics}

This section contains a collection of rather technical lemmas necessary
to show that the first return map on $\Sigma$ in Lemma \ref{lem:first_return_map_as_Fqxk}
is given by powers of $\Fqx$.

\begin{lem}
\label{lem:if_intersects_a0minus1}If $\gamma\in\Rgeo$ with $\nlm{\gamma_{+}}=a_{0}$
intersects $T^{a_{0}-1}L_{0}$ then \[
\gamma_{+}=\begin{cases}
\RS{a_{0};\left(1^{h-1},2\right)^{l},1^{h},a_{\left(l+1\right)h+1},\ldots} & \mbox{for some\,\,}0\le l\le\infty,\, a_{\left(l+1\right)h+1}\le-1,\,\mbox{if}\,\, q\,\mbox{is even}\\
\RS{a_{0};1^{h},a_{h+1},\ldots}, & \mbox{with}\,\, a_{h+1}\ge2,\,\mbox{if\,\,}q\,\mbox{is odd.}\end{cases}\]

\end{lem}
\begin{proof}
Let $x=\gamma_{+}-a_{0}\lambda$. By convexity $\gamma$ does not
intersect $T^{a_{0}-1}\overline{L_{0}}$ more than once. Hence $x\in\left(-\frac{\lambda}{2},1-\lambda\right]$.
If $q$ is odd, then $\Cq\left(1\right)=\RS{1;1^{h}}$ and $\Cq\left(-\frac{\lambda}{2}\right)=\RS{1^{h},2,1^{h}}$
according to Lemmas \ref{lem:expansion_of_1_odd_q} and \ref{lem:lambda_half_finite_expansion}
hence $\RS{1^{h},2,1^{h}}<\Cq\left(x\right)<\RS{1^{h}}$ and by the
lexicographic ordering (see proof of Lemma \ref{lem:R-is-given-by-lambda-fraction})
it is clear, that $\Cq\left(x\right)=\RS{1^{h},a_{h+1},\ldots}$ with
$a_{h+1}\ge2$. If $q$ is even then $\Cq\left(-\frac{\lambda}{2}\right)=\RS{1^{h}}$
and $1-\lambda=r$ with $\Cq\left(r\right)=\RS{\overline{1^{h-1},2}}$
so that $\RS{1^{h}}<\Cq\left(x\right)\le\RS{\overline{1^{h-1},2}}$.
By the lexicographic ordering it is clear, that $\Cq\left(x\right)=\RS{\left(1^{h-1},2\right)^{l},1^{h},a_{\left(l+1\right)h+1},\ldots}$
for some $l\ge0$ ($l=\infty$ is allowed) and $a_{\left(l+1\right)h+1}\le-1$
if $l<\infty$. 
\end{proof}
\begin{lem}
\label{lem:translates_not_reduced}Let $\gamma\in\Rgeo$ with $\nlm{\gamma_{+}}=a_{0}$.
Then $T^{-\sign\left(a_{0}\right)n}\gamma$ is not reduced for $n\ge1$. 
\end{lem}
\begin{proof}
Without loss of generality assume $a_{0}\ge1$ and let $\gamma^{n}=T^{-n}\gamma$
for $n\ge1$. Then $\gamma_{+}^{n}=\gamma_{+}+n\lambda>\gamma_{+}$
and $\gamma_{-}^{n}=\gamma_{-}+n\lambda$. Since $\gamma$ is reduced
$\gamma_{-}\ge-R$ $\Rightarrow$ $\gamma_{-}^{n}\ge n\lambda-R=-r+\left(n-1\right)\lambda\ge-r$.
Hence $\gamma_{-}^{n}\notin\left[-R,-r\right)$ so $\gamma^{n}\notin\Rgeo$.
The case of $a_{0}\le-1$ is analogous. 
\end{proof}
\begin{lem}
\label{lem:pullback_a0_unique}If $\gamma\in\Rgeo$ then $\gamma^{n}=ST^{-n}\gamma$
is reduced if and only if $n=a_{0}=\nlm{\gamma_{+}}$.
\end{lem}
\begin{proof}
By definition, if $\gamma^{n}\in\Rgeo$ then $S\gamma_{+}^{n}=\gamma_{+}-n\lambda\in I_{q}^{\infty}$
and hence $\gamma_{+}\in\left(n\lambda-\frac{\lambda}{2},n\lambda+\frac{\lambda}{2}\right)$
$\Rightarrow$ $n=a_{0}.$ It is also clear that $\gamma^{a_{0}}=\Fqx\gamma$
is reduced. 
\end{proof}
\begin{lem}
\label{lem:even_crossing_dj_gives_no_return}Suppose that $q$ is
even. If $\gamma=\gamma\left(\xi,\eta\right)\in\Rgeo$ with $a_{0}=\nlm{\gamma_{+}}\ge2$
intersects $T^{a_{0}-1}L_{0}$ then the first return map is given
as $\RETx\left(\xi,\eta\right)=\Fqx\,^{\K\left(\xi\right)}\left(\xi,\eta\right)\in\PT^{-1}\left(\Sigma^{\n\left(\xi\right)}\right)$
where $\K\left(\xi\right)$ and $\n\left(\xi\right)$ are as in Definition
\ref{def:Define_K}. 
\end{lem}
\begin{proof}
Consider Figures \ref{fig:Geodesics-from-rho} and \ref{fig:Geodesics-returning-even}
showing the arcs around the point $\rho$. The picture is symmetric
with respect to $\Re z=\frac{\lambda}{2}$ and invariant under translation,
so it applies in the present case. After passing through $T^{a_{0}-1}L_{0}$
the geodesic $\gamma$ will intersect a sequence of translates of
the arcs $\overline{\chi_{j}}$ and $\overline{\omega_{j}}$ which
are the reflections of $\chi_{j}$ and $\omega_{j}$ in $\Re z=\frac{\lambda}{2}$
exactly as in Lemma \ref{lem:first_return_map_as_Fqx^j+1_as_a0=00003D1},
except that it now passes through every arc. Note, even the argument
why $\gamma$ intersects $\overline{\omega}_{j}$ and not its extension
applies. Let $w_{n}$ and $A_{n}$ be as in Definition \ref{def:First_return_map}
except that now $w_{0}\in T^{a_{0}-1}L_{0}$. 

Then $w_{2j}\in T^{a_{0}-1}\overline{\omega}_{j}=T^{a_{0}-1}\left(ST^{-1}\right)^{j}SL_{0}$,
for $0\le j\le h$ and also $w_{2j+1}\in T^{a_{0}-1}\overline{\chi}_{j}=\left(ST^{-1}\right)^{j+1}L_{1}$,
for $0\le j\le h$ with corresponding maps $A_{2j+1}=\left(TS\right)^{j+1}T^{1-a_{0}}$
and $A_{2j}=\left(ST\right)^{j}ST^{1-a_{0}}$. Set $\gamma_{k}:=A_{k}\gamma$
and $\xi_{k}:=A_{k}\xi$ for $0\le k\le2h+1$. There are three cases
when the point $w_{k}$ can define a return to $\Sigma$: 
\begin{lyxlist}{00.00.0000}
\item [{a)}] if $\gamma_{2j+1}\in\Sgeo$ and $z=\PP\gamma_{2j+1}\in\Sigma^{1}$, 
\item [{b)}] if $\gamma_{2j-1}\notin\Sgeo$ but $\gamma'_{2j-1}=TS\gamma_{2j-1}\in\Sgeo$
with $\xi_{2j-1}'\in\left(\frac{3\lambda}{2},\lambda+1\right)$ and
$z=\PP\gamma_{2j-1}'\in\Sigma^{2}$, or 
\item [{c)}] if$\gamma_{2j}\in\Sgeo$ and $z=\PP\gamma_{2j}\in\Sigma^{0}$.
\end{lyxlist}
According to Lemma \ref{lem:if_intersects_a0minus1} we get $T^{a_{0}-1}\xi\in\left(\frac{\lambda}{2},1\right]$
and $\Cq\left(\xi\right)=\RS{a_{0};1^{h-1},a_{h},a_{h+1},\ldots}$
with $a_{h}=2$ or $a_{h}=1$ and $a_{h+1}\le-1$. Also $\frac{\lambda}{2}=\RS{\left(-1\right)^{h}}$
and $1=R=\DS{\left(-1\right)^{h},\overline{-2,\left(-1\right)^{h-1}}}$. 

Hence $\xi_{2j+1}\in\left(-\phi_{j+1},-r_{h-j}\right]\subseteq I_{q}\Rightarrow\gamma_{2j+1}\notin\Rgeo$
for $1\le j\le h-1$ (cf.~Remark \ref{rem:For-even-q-definitions}).
Also note, that $A_{2h+1}=\left(TS\right)^{h+1}T^{1-a_{0}}=\left(ST^{-1}\right)^{h}ST^{-a_{0}},$
therefore $\gamma_{2h+1}=A_{2h+1}\gamma=ST^{-1}\Fqx\,^{h}\gamma$
and hence $\gamma_{2h+1}\notin\Rgeo$ unless $a_{h}=1$ in which case
$\gamma_{2h+1}=\Fqx\,^{h+1}\gamma\in\Sgeo$ and $\PP\gamma_{2h+1}\in\Sigma^{1}$. 

For $\gamma_{2j+1}'=TS\gamma_{2j+1}=\gamma_{2j+3}$ we conclude $\gamma_{2j+1}'\notin\Rgeo$
for $0\le j\le h-2$. For $a_{h}=1$ we find $\gamma_{2h-1}'=\gamma_{2h+1}\in\Sgeo$,
but $\xi_{2h-1}'<0$ and hence we do not get a point of $\Sigma^{2}$.
If $a_{h}=2$ on the other hand then $\gamma_{2h-1}'\notin\Sgeo$
but $\gamma_{2h+1}'=TS\gamma_{2h+1}=\Fqx\,^{h}\gamma\in\Sgeo$ and
$\xi'_{2h+1}>\frac{3\lambda}{2}$ so $\PP\gamma_{2h+1}'\in\Sigma^{2}$. 

For $\gamma_{2j}=S\gamma_{2j-1}$ we get $\xi_{2j}\in S\left(-\phi_{j},-r_{h+1-j}\right)=\left(-\lambda-\phi_{j+1},-\lambda-r_{h-j}\right)$
and therefore $-\frac{3\lambda}{2}<\xi_{2j}<-\frac{\lambda}{2}$ for
$0\le j\le h-1$. Since $T^{1-a_{0}}\eta<-R$ we have $\eta_{2j}\in S\left(TS\right)^{j}\left(-\infty,-R\right)=\left(\phi_{h-j},r_{j}\right)$
that is $\eta_{2j}<0$. Hence $\gamma_{2j}\notin\Sgeo$ for $0\le j\le h-1.$
Note, that $\Fq^{h}\xi>\frac{2}{\lambda}$ implies that $\xi_{2h}=T^{-1}ST^{-1}\Fq^{h}\xi\in T^{-1}ST^{-1}\left(\frac{2}{\lambda},\infty\right)=\left(-\lambda,-\lambda+\frac{\lambda}{\lambda^{2}-2}\right)$
for $q>4$ and $\xi_{2h}\in\left(-\infty,-\lambda\right)$ for $q=4$.
In any case $\xi_{2h}<0$ and since $\eta_{2h}<r$ it is clear, that
$\gamma_{2h}\notin\Rgeo$.

We conclude that the first return is given by $w_{2h+1}$ and $\RETx\left(\xi,\eta\right)=\Fqx^{h+1}\left(\xi,\eta\right)\in\PT^{-1}\left(\Sigma^{1}\right)$
if $a_{h}=1$ and $\RETx\left(\xi,\eta\right)=\Fqx^{h}\left(\xi,\eta\right)\in\PT^{-1}\left(\Sigma^{2}\right)$
if $a_{h}=2$. This can be written in the form $\RETx=\Fqx^{\K\left(\xi\right)}\in\PT^{-1}\left(\Sigma^{\n\left(\xi\right)}\right)$
with $\K\left(\xi\right)$ and $\n\left(\xi\right)$ as in Definition
\ref{def:Define_K}.
\end{proof}
\begin{lem}
\label{lem:even_boundary_ah=00003D1}For $q$ even consider $\xi\in\left(-\frac{\lambda}{2},1-\lambda\right)$
with $\Cq\left(\xi\right)=\RS{0;1^{h-1},a_{h},a_{h+1},\ldots}.$ Then
$a_{h}=1$ if and only if $\xi<\frac{-\lambda^{3}}{\lambda^{2}+4}=-\lambda+\frac{4\lambda}{\lambda^{2}+4}$. 
\end{lem}
\begin{proof}
By Lemma \ref{lem:if_intersects_a0minus1} we know either $a_{h}=1$
and $a_{h+1}\le-1$ or $a_{h}=2$. It is clear that the boundary point
between these two cases is given by $\xi_{0}=\RS{0;1^{h},\left(-1\right)^{h}}=\left(ST\right)^{h}\left(\frac{\lambda}{2}\right)$.
Using (\ref{eq:TS-explicit}) one can show that $\xi_{0}=\left(ST\right)^{h}\left(\frac{\lambda}{2}\right)=\frac{\left(\lambda^{2}-2\right)\frac{\lambda}{2}+\lambda}{-\lambda\frac{\lambda}{2}-2}=\frac{-\lambda^{3}}{4+\lambda^{2}}.$ 
\end{proof}
The following corollary is easy to verify by estimating the intersection
of $\gamma\left(-r,\left(a_{0}-\frac{1}{2}\right)\lambda\right)$
and $T^{a_{0}-1}L_{0}$. It implies that the case $\RET=\Fqx^{h}$
does not occur for $\left\{ \xi\right\} _{\lambda}\ge3$. 

\begin{cor}
Let $q$ be even and suppose that $\gamma=\gamma\left(\xi,\eta\right)\in\Rgeo$
with $a_{0}=\left\{ \xi\right\} _{\lambda}\ge3$. If $\gamma$ intersects
$T^{a_{0}-1}L_{0}$ then $\Cq\left(\xi\right)=\RS{a_{0};1^{h},a_{h+1},\ldots}$
with $a_{h+1}\le-1$. 
\end{cor}
\begin{lem}
\label{lem:odd_reduced_does_not_intersect_l}Let $q$ be odd and suppose
that $\gamma\in\Rgeo$. Let $l$ be the geodesic arc $\left[\rho+\lambda,1+\lambda\right]$,
i.e. the continuation of $L_{3}$. Then $\gamma$ does not intersect
$\pm l$ outwards (i.e. in the direction from $0$ to $\pm\infty$). 
\end{lem}
\begin{proof}
Take $\gamma=\gamma\left(\xi,\eta\right)\in\Rgeo$ and assume $\xi>0$.
Suppose that $\gamma$ intersects $l$ in the outwards direction.
Since $\gamma$ can not intersect the geodesic $TS\overline{L}_{-1}=\left[\lambda,\lambda+\frac{2}{\lambda}\right]$
more than once we have $\xi\in\left(\lambda+1,\lambda+\frac{2}{\lambda}\right)$
and because $\gamma\in\Rgeo$ we have $-R\le\eta<-r$. If $w\left(\xi,\eta\right)$
is the intersection between $\gamma$ and the line $T\overline{L}_{1}=\frac{3\lambda}{2}+i\mathbb{R}^{+}$
then $\Im w\left(\xi,\eta\right)>\Im w\left(\xi,-r\right)\ge\Im w\left(\lambda+1,-r\right)$.
To show that $\Im w>\Im T\rho=\Im\rho=\sin\frac{\pi}{q}$ it is enough
to bound $\Im w\left(\lambda+1,-r\right)$ from below. By Lemma \ref{lem:intersection_vertical_and_circular}
we have \begin{align*}
\Im w\left(\lambda+1,-r\right)^{2} & =\left(\lambda+1-\frac{3\lambda}{2}\right)\left(\frac{3\lambda}{2}-r\right)=\left(1-\frac{\lambda}{2}\right)\left(\frac{5\lambda}{2}-R\right)\\
 & =\left(1-\frac{\lambda}{2}\right)\left(1+\frac{\lambda}{2}+\left(2\lambda-R-1\right)\right)\\
 & =\sin^{2}\frac{\pi}{q}+\left(1-\frac{\lambda}{2}\right)\left(2\lambda-R-1\right)>\sin^{2}\frac{\pi}{q}\end{align*}
since $2\lambda>R+1$ and $1-\frac{\lambda}{2}>0$. Hence $\Im w\left(\xi,\eta\right)>\sin\frac{\pi}{q}$
and $\gamma$ does not intersect $l$ in the direction from $0$ to
$\infty$. An analogous argument for $\xi<0$ concludes the Lemma.
\end{proof}
\begin{lem}
\label{lem:odd_crossing_dj_gives_no_return}For $q$ odd, let $\gamma=\gamma\left(\xi,\eta\right)\in\Sgeo$
be strongly reduced with $a_{0}=\left\{ \xi\right\} _{\lambda}\ge2$.
If $\gamma$ intersects $T^{a_{0}-1}L_{0}$ then $\RETx\left(\xi,\eta\right)=\Fqx^{h+1}\left(\xi,\eta\right)\in\PT^{-1}\left(\Sigma^{3}\right)$.
\end{lem}
\begin{proof}
Consider once more Figure \ref{fig:Geodesics-from-rho} showing the
arcs around $\rho$. Analogous to the proof of Lemma \ref{lem:even_crossing_dj_gives_no_return}
we have $w_{2j}\in T^{a_{0}-1}\overline{\omega}_{j}=T^{a_{0}-1}\left(ST^{-1}\right)^{j}SL_{0}$,
for $0\le j\le h+1$ and $w_{2j+1}\in T^{a_{0}-1}\overline{\chi}_{j}=T^{a_{0}-1}\left(ST^{-1}\right)^{j+1}L_{1}$,
for $0\le j\le h$ with the corresponding maps $A_{2j+1}=\left(TS\right)^{j+1}T^{1-a_{0}}$
and $A_{2j}=\left(ST\right)^{j}ST^{1-a_{0}}$. Set $\gamma_{j}:=A_{j}\gamma$
and $\xi_{j}:=A_{j}\xi$. There are now four possibilities to produce
a return to $\Sigma$: 
\begin{lyxlist}{00.00.0000}
\item [{a)}] if $\gamma_{2j+1}\in\Sgeo$ and $\mathbf{z}=\PP\gamma_{2j+1}\in\Sigma^{1}$,
\item [{b)}] if $\gamma_{2j-1}\notin\Rgeo$ but $TS\gamma_{2j+1}\in\Sgeo$
and $\mathbf{z}=\PP TS\gamma_{2j+1}\in\Sigma^{2}$,
\item [{c)}] if $\gamma_{2j}\in\Sgeo$ and $\mathbf{z}=\PP\gamma_{2j}\in\Sigma^{0}$,
\item [{d)}] if $\gamma_{2j}\notin\Rgeo$ but $T^{\pm1}\gamma_{2j}\in\Sgeo$
and $\mathbf{z}=\PP\gamma_{2j}\in\Sigma^{\pm3}$. 
\end{lyxlist}
We will see that most of these cases do not give a return. Since $T^{1-a_{0}}\xi\in\left(\frac{\lambda}{2},1\right)$
Lemma \ref{lem:if_intersects_a0minus1} shows that $\Cq\left(\xi\right)=\RS{a_{0};1^{h},a_{h+1},\ldots}$
with $a_{h+1}\ge2$. Suppose also, that $\Cd\left(\eta\right)=\DS{0;b_{1},b_{2},\ldots}$.
For the following arguments it is important to remember that the action
of $TS$ on $\partial\H\cong\mathbb{R}^{*}\cong S^{1}$ is monotone
as a rotation around $\rho$. 

Since $\gamma_{2j+1}=\left(TS\right)^{j+1}T^{1-a_{0}}\gamma$ we have
$\xi_{2j+1}\in\left(TS\right)^{j+1}\left(\frac{\lambda}{2},1\right)=\left(-\phi_{2j+2},-\phi_{2j+1}\right)\subseteq I_{q}$
and hence $\gamma_{2j+1}\notin\Rgeo$ for $0\le j\le h-1$. Furthermore
$\xi_{2h+1}\in TS\left(-\phi_{2h},-\phi_{2h+1}\right)=\left(-1,0\right)$
and therefore $\left|\xi_{2h+1}\right|<\frac{2}{\lambda}$ so that
also $\gamma_{2h+1}\notin\Rgeo$.

If $\gamma'_{2j+1}=TS\gamma_{2j+1}=\gamma_{2j+3}$ then $\gamma'_{2j+1}\notin\Rgeo$
for $1\le j\le h-1$ and since we have $A_{2h+3}=\left(TS\right)^{h+2}T^{1-a_{0}}=\left(ST^{-1}\right)^{h}ST^{-a_{0}}$
it is clear that $\gamma_{2h+3}=A_{2h+3}\gamma=\Fqx^{h+1}\gamma\in\Sgeo$
and we have a return at $w_{2h+1}$ with $\mathbf{z}_{1}\in L_{2}$! 

Since $\gamma_{2j}=\left(ST\right)^{j}ST^{1-a_{0}}\gamma$ obviously
$\xi_{2j}\in S\left(TS\right)^{j}\left(\frac{\lambda}{2},1\right)=\left(S\left(-\phi_{2j}\right),S\left(-\phi_{2j-1}\right)\right)=-\lambda-\left(\phi_{2j+2},\phi_{2j+1}\right)$
and hence $-\frac{3\lambda}{2}<\xi_{2j}<-\frac{\lambda}{2}$. But
$T^{1-a_{0}}\eta<-r-\lambda=-R$ and hence $\eta_{2j}\in\left(ST\right)^{j}S\left(-\infty,-R\right)=\left(\phi_{2\left(h-j\right)+1},r_{2j+1}\right)$
for $0\le j\le h$ (cf.~Remark \ref{rem:For-odd-q-definitions})
respectively $\eta_{2h+2}\in ST\left(\phi_{1},r_{2h+1}\right)=ST\left(1-\lambda,r\right)=S\left(1,R\right)=\left(-1,-1/R\right)$
and therefore $\eta_{2j}<0$ and $\gamma_{2j}\notin\Sgeo$ for $0\le j\le h+1$.

Finally, since $T^{-1}\eta_{2j}<r_{2j+1}-\lambda<-R$ and $T\xi_{2j}\in I_{q}$
it is clear that $T^{\pm1}\gamma_{2j}\notin\Rgeo$ for $0\le j\le h+1$.

\end{proof}

\bibliographystyle{plain}
\bibliography{/home/fredrik/Documents/matematik/refs}

\end{document}